\documentclass[11pt]{article}
\usepackage{amssymb, latexsym, mathtools, amsthm}
\begingroup
    \makeatletter
    \@for\theoremstyle:=definition,remark,plain\do{%
        \expandafter\g@addto@macro\csname th@\theoremstyle\endcsname{%
            \addtolength\thm@preskip\parskip
            }%
        }
\endgroup
\usepackage{enumerate}
\usepackage{pgf,tikz}
\usepackage{pdfsync}
\usepackage{natbib}
\usepackage{txfonts}
\usepackage[T1]{fontenc}
\usepackage{booktabs}
 \usepackage{graphicx}
\definecolor{dnrbl}{rgb}{0,0,0.3}
\definecolor{dnrgr}{rgb}{0,0.3,0}
\definecolor{dnrre}{rgb}{0.5,0,0}
\usepackage[colorlinks=true, citecolor=dnrgr, linkcolor=dnrre, urlcolor=dnrre] {hyperref}
\usepackage{xcolor}
\usepackage{parskip}
\usepackage{tabularx, colortbl}
\usepackage{geometry}
\theoremstyle{plain}
\newtheorem{thm}{Theorem}[section]

\newtheorem{lem}[thm]{Lemma}
\newtheorem{coro}[thm]{Corollary}

\theoremstyle{definition}

\newtheorem{defi}[thm]{Definition}
\makeatletter

\let\c@table\c@figure
\makeatother

\newcommand{\Nat}{\mathbb{N}}

\newcommand{\restr}{\upharpoonright}  
 
\newcommand{\un}{\uparrow} 
\newcommand{\de}{\downarrow} 

\newcommand{\bigo}[1]{\mathop{\bf O}\/\left({#1}\right)}

\newcommand{\wgt}[1]{\mathop{\mathtt{wgt}}\/\left({#1}\right)}

\newcommand{\wgtb}[1]{\mathop{\mathtt{wgt}}\/\big({#1}\big)}

\newcommand{\ml}{Martin-L\"{o}f }

\newcommand{\ie}{i.e.\ }
\newcommand{\ce}{c.e.\ }

\newcommand{\lce}{left-c.e.\ }

\newcommand{\pf}{prefix-free }

\newcommand{\twomel}{2^{<\omega}}

\renewenvironment{abstract}
 { \normalsize
  \list{}{
    \setlength{\leftmargin}{.0cm}%
    \setlength{\rightmargin}{\leftmargin}%
    }%
  \item {\bf \abstractname.} \relax}
 {\endlist}
\usepackage{titlesec}
 \titleformat{\paragraph}
{\normalfont\normalsize\em\bfseries}{\theparagraph}{1em}{$\blacktriangleright$\hspace{0.2cm}}
\titlespacing*{\paragraph}
{0pt}{3.25ex plus 1ex minus .2ex}{0.5ex plus .2ex}

\title{Monotonous betting strategies in warped casinos
\thanks{Barmpalias was supported by the 
1000 Talents Program for Young Scholars from the Chinese Government No.~D1101130, 
NSFC grant No.~11750110425 and Grant No.~ISCAS-2015-07 from the Institute of Software.
Fang Nan was supported by the China Scholarship Council of the Ministry of Education of China.}}

\author{George Barmpalias  \and Nan Fang \and Andrew Lewis-Pye}
\date{\today}
\begin{document}
\maketitle
\begin{abstract}
Suppose that the outcomes of a roulette table  are not entirely random, in the sense that
there exists a successful betting strategy.
Is there a successful `separable' strategy, in the sense that it does not use
the winnings from betting on red in order to bet on black, and vice-versa? 
We study this question from an algorithmic point of view and observe that
every strategy $M$ can be replaced by a separable strategy
which is computable from $M$ and successful on any outcome-sequence where $M$ is successful.
We then consider the case of mixtures and show:
(a) there exists an effective mixture of separable strategies which succeeds on every casino sequence 
with effective Hausdorff dimension less than 1/2; (b) there exists a casino sequence of 
effective Hausdorff dimension 1/2 on which no effective mixture of separable strategies succeeds.
Finally we extend (b) to a more general class of strategies.
 \end{abstract}
\vspace*{\fill}
\noindent{\bf George Barmpalias}\\[0.2em]
\noindent
State Key Lab of Computer Science, 
Institute of Software, Chinese Academy of Sciences, Beijing, China.\\[0.2em] 
\textit{E-mail:} \texttt{\textcolor{dnrgr}{barmpalias@gmail.com}}.
\textit{Web:} \texttt{\textcolor{dnrre}{http://barmpalias.net}}\par
\addvspace{\medskipamount}\medskip\medskip

\noindent{\bf Nan Fang}\\[0.2em]
\noindent Institut f\"{u}r Informatik, Ruprecht-Karls-Universit\"{a}t Heidelberg, Germany.\\[0.1em]
\textit{E-mail:} \texttt{\textcolor{dnrgr}{nan.fang@informatik.uni-heidelberg.de}.}
\textit{Web:} \texttt{\textcolor{dnrre}{http://fangnan.org}} \par
\addvspace{\medskipamount}\medskip\medskip

\noindent{\bf Andrew Lewis-Pye}\\[0.2em]  
\noindent Department of Mathematics,
Columbia House, London School of Economics, 
Houghton St., London, WC2A 2AE, United Kingdom. \\
\textit{E-mail:} \texttt{\textcolor{dnrgr}{A.Lewis7@lse.ac.uk.}}
\textit{Web:} \texttt{\textcolor{dnrre}{http://lewis-pye.com.}} 

\vfill \thispagestyle{empty}
\clearpage

\section{Introduction}\label{sQtcO25fky}\label{VYuXf7zTQg}
A bet in a game of chance is usually determined by two values: the favorable outcome and the wager $x$ one bets on
that outcome. If the outcome turns out to be the one chosen, the player gains profit $x$;  otherwise the player loses the wager $x$.
Many gambling systems for repeated betting are based on elaborate choices for the wager $x$,
while leaving the choice of outcome constant. In this work we are interested in
such `monotonous' strategies, which we also call {\em single-sided},
and their linear combinations (mixtures). 
Consider the game of roulette, for example,  and the binary outcome of red/black.\footnote{Roulettes
have a third outcome 0, which is neither red nor black, and which gives a slight advantage to the house. 
For simplicity in our discussion we
ignore this additional outcome.} Perhaps the most infamous roulette system 
is the {\em martingale},\footnote{for the origin of this term, its use as a betting system and its 
adoption in mathematics, see \cite{histmartinguse} and \cite{martmathintel}.} 
where one constantly bets on a fixed color,
say {\em red}, starts with an initial wager $x$ 
and doubles the wager after each loss. At the first winning stage all losses 
are then recovered and an additional profit $x$ is achieved. Such systems rely on  the fairness of the game, 
in the form of a law of large numbers that has to obeyed in the 
limit (and, of course, require unbounded initial resources in order to 
guarantee success with probability 1). In the example of the martingale the relevant law is that, with probability 1, 
there must be a round where the outcome is red. Many other systems have been developed 
that use more tame series of wagers (compared to the exponential 
increase of the martingale), and which appeal to various forms of the law of large numbers.\footnote{Well-known systems 
of this kind are: the D'Alembert System, the Fibonacci system, the Labouch\`{e}re system or split martingale, and many others. 
See, for example, \texttt{https://www.roulettesystems.com}.}

When the casino is biased, \ie the outcomes are not entirely random, we ought to be able to produce
more successful strategies. Suppose that we bet on repeated  coin-tosses, and
that we are given the information that the coin has a bias. 
In this case it is well known that we can define an effective strategy that, independent of the
bias of the coin (\ie which side the coin is biased on, or even any lower bounds on the bias), is guaranteed to gain
unbounded capital, starting from any non-zero initial capital. This strategy, as we explain in
\S \ref{OC53m9ZSeG}, is the mixture of two single-sided strategies, where
the first one always bets on heads and the second one always bets on tails.
A slightly modified strategy is successful on every coin-toss sequence $X$ except for the case that
the limit of the relative frequency of heads exists and is 1/2.
The same kind of strategy exists for the case where the
relative frequency of heads is 1/2, but beyond some point the number of tails is never smaller than
the number of heads (or vice-versa). These examples show that
many typical betting strategies are {\em separable} in the sense that
they can be expressed as a the sum of two single-sided strategies.
In the following we refer to any binary sequence which is produced 
by a (potentially partially) random process, as a {\em casino sequence}.
Note that if a separable strategy succeeds along a casino sequence, one of its
single-sided parts  has to succeed. The only case where separability is
stronger than single-sidedness is when we consider success with respect to classes of casino sequences.

A casino sequence may have a (more subtle) bias while satisfying several known laws of large numbers, such as
the relative frequency of 0s tending to 1/2. Formally, we can say that a casino sequence $X$ is biased if
there is an `effective' (as in `constructive' or `definable') 
betting strategy which succeeds on $X$, \ie produces an unbounded capital, starting from a finite initial capital. 
By adopting stronger or weaker formalizations
of the term `effective' one obtains different strengths of bias, or as we usually say, {\em non-randomness} of $X$.
In general, `effective' means that the strategy is definable in a simple way, such as being  programmable in a 
Turing machine.
%
Suppose that we know that the casino sequence $X$ has a bias in this more general sense, \ie 
 there exists some `effective' betting strategy which succeeds on it. 
The starting point of the present article is the following question:
\begin{equation}\label{ID1kUpxqvY}
\parbox{12.5cm}{\em Is it possible to succeed on any such warped casino sequence
with a single-sided `effective' betting strategy, \ie one that can only place bets on 0 
or  only  on 1?}
\end{equation}
In other words, can any `effective' betting strategy be replaced by a single-sided
`effective' betting strategy without sacrificing success? An equivalent way to ask this question is as follows.
\begin{equation}\label{ZrEv5hZl2}
\parbox{14.5cm}{\em Suppose that we are betting with 
the restriction that we cannot use our earnings from the
successful bets on 0s in order to bet on 1s, 
and vice-versa. Can we win on any  casino-sequence $X$ which is `biased' in the 
sense that there is an (unrestricted) strategy which wins on $X$?}
\end{equation}
We will see that, depending on the way we formalise the term `effective', and especially the term
{\em effective monotonous betting} these questions can have a positive or negative (or even unknown) answer.

{\bf Our results.}
A straightforward interpretation of `effective' is computable, in the sense that there is a Turing machine
that decides, given  each initial segment of the casino sequence: 
\begin{equation*}
\parbox{12cm}{(a) how much of the current capital to bet; \hspace{0.3cm}(b) which outcome to bet on.}
\end{equation*}
%
%
These choices, in  combination with the revelation of the outcome, determine
the capital at the beginning of the next betting stage.
In \S\ref{aRvZodD8FN} we show that in this case questions \eqref{ID1kUpxqvY} and \eqref{ZrEv5hZl2} 
have a positive answer.
Another formalisation of `effective'  which is very standard in computability and algorithmic information theory
(and used in the standard definition of algorithmic randomness) is `computably enumerable'. 
When applied to betting strategies this gives a notion which is equivalent
to infinite mixtures of strategies which are
generated by a single Turing machine, see the introductory part of \S\ref{a9WlWLzsyd}.
There are two very different ways that one can define computably enumerable monotonous
strategies:
\begin{enumerate}[\hspace{0.3cm}(i)]
\item {\em Uniform way:} as the mixture (linear combination) of a computable {\em family of monotonous strategies} 
with bounded total initial capital;
\item {\em Non-uniform way:} as a {monotonous strategy} that can be expressed 
as the mixture of a computable family of  strategies with bounded total initial capital.
\end{enumerate}
In the uniform case we show that
questions \eqref{ID1kUpxqvY} and \eqref{ZrEv5hZl2} have negative answers. In fact, we show that
there are casino sequences $X$ on which mixtures of computable families of strategies
generate infinite capital exponentially fast, in the sense 
that\footnote{.} 
\begin{equation}\label{q4Zg8ozv2}
\limsup_n \frac{M(X\restr_n)}{\alpha^n}=\infty
\hspace{0.5cm}\textrm{where $\alpha\in (1,\sqrt{2})$ and $M$ is the capital after the first $n$ bets on $X$,}
\end{equation}
where $X\restr_n$ denotes the first $n$ bits of $X$, 
but no strategy under (i) succeeds.
We also show the converse, \ie that if a computably enumerable strategy 
(\ie a mixture of computable family of strategies)
$M$ exists such that
$\limsup_n M(X\restr_n)/\alpha^n=\infty$ for some
$\alpha>\sqrt{2}$, then there exists a single-sided computably enumerable strategy $N$
which succeeds on $X$, in the sense that 
$\lim_n N(X\restr_n)=\infty$.
We will see that these results can also be stated in terms of the effective Hausdorff dimension of 
the casino sequence.
Under the uniform case we also
consider a more general class of strategies, which we call {\em decidably-sided}, and
which are not necessarily monotonous,  but there is a computable prediction (or choice) function which 
indicates the favorable outcome at each state. We then
generalise our previous arguments  and show that there is a casino sequence
and a computably enumerable betting strategy $M$ that strongly succeeds on it as before, in the sense of \eqref{q4Zg8ozv2}, but such that 
no decidably-sided computably enumerable strategy succeeds on it.

Monotonous strategies under the non-uniform clase (ii) are intuitively more powerful, as we explain in \S\ref{6E1fkvgPTH},
and our arguments do not appear to be adequate for answering 
questions  \eqref{ID1kUpxqvY} and \eqref{ZrEv5hZl2}  in this case.
The study of the power of strategies in (ii) is quite interesting from the point of view of 
stochastic processes, as it relates to key concepts such as martingale decompositions, variation and various forms of
boundedness or integrability. Questions \eqref{ID1kUpxqvY} and \eqref{ZrEv5hZl2}  under (ii) are also directly relevant to
a question about the separation of two randomness notions in algorithmic information theory, 
asked by Kastermans (see \citep{Downey:2012:RCM:2367234} and \citep[\S 7.9]{rodenisbook}).
As we point out in \S \ref{7FjYtfZZf}, a positive answer of \eqref{ID1kUpxqvY} or \eqref{ZrEv5hZl2}
for the case of strategies under (ii) would give a very simple and elegant positive answer to Kasterman's question.

{\bf Outline of the presentation.}
The concept of a betting strategy in terms of 
martingale functions
is formalised in the first part of \S \ref{a9WlWLzsyd}. Monotonous strategies are formalised 
in \S\ref{JbvJlths9B} and effective versions of mixtures of monotonous strategies are given in 
\S\ref{6E1fkvgPTH}, along with relevant characterizations in terms of computable enumerability.
In \S\ref{OC53m9ZSeG} we show that many types of betting are monotonous and in  
\S\ref{Sp9icqFzvI}, after recalling that Hausdorff dimension is expressible in terms of speed of martingale success, 
we use these facts in order to show that there exists a separable strategy
which succeeds in all casino sequences of effective Hausdorff dimension $<1/2$.
In \S \ref{aRvZodD8FN} we first describe a decomposition of computable martingales into
two single-sided (orthogonal) martingales, which
provides the positive answer to questions \eqref{ID1kUpxqvY} and \eqref{ZrEv5hZl2} 
stated in the introductory discussion, for the case of computable strategies.
We then give a detailed argument establishing a strong negative answer of the same questions
for the special case of a
single separable strategy.
This argument is then used in a modular way in
\S\ref{BBcmfYbEE} in order to obtain a proof of the full result, with respect to every possible
strategy that is expressible as a mixture of a computable family of separable martingales.
Finally in \S\ref{meVVbZop96} we generalize this result to the more general class of decidably-sided strategies.
Concluding remarks and a critical discussion of our results, along with open problems and directions for future investigations 
are given in \S\ref{7FjYtfZZf}.

\section{Monotonous betting strategies and their mixtures}\label{a9WlWLzsyd}
Betting strategies are formalized by 
martingales\footnote{This is a mathematical notion and different than the martingale 
betting system that we discussed in \S \ref{sQtcO25fky}. In mathematics, martingales 
were introduced by \cite{Levymartin} and extended
by \cite{villemartin} who also gave them this name. See \cite{doobmartinAMS} for a classic and brief exposition of martingales
in probability.} 
which are used in order to express the capital after each betting stage and each casino outcome.
Formally, a {\em martingale} in the space of binary outcomes 
is a function $M:\twomel\to\mathbb{R}^{\geq 0}$ 
from binary strings to the non-negative real numbers, 
with the property that for all $\sigma\in 2^{<\omega}$:
\begin{equation}\label{TwlwwLBPiK}
2\cdot M(\sigma)=M(\sigma\ast 0)+M(\sigma\ast 1).
\end{equation}
If the equality
is replaced with `$\geq$' then $M$ is called a {\em supermartingale}.\footnote{Supermartingales 
can be viewed as `leaky martingales' which
may potentially lose some capital at each betting position.}
 Probabilistically, such a function $M$ can be
seen as a martingale stochastic process $(Y_s)$ {\em relative} to the 
underlying fair coin-tossing stochastic process
$(I_s)$, where $I_s$ is the outcome of the $s$th coin-toss which can be 0 or 1 
with equal probability 1/2, so that:
\begin{enumerate}[\hspace{0.5cm}(a)]
\item $Y_s$ is measurable in (\ie determined by the outcome of) $I_i, i\leq s$;
\item by \eqref{TwlwwLBPiK} the expectation of $Y_{s+1}$ given $I_i, i\leq s$ equals $Y_s$.
\end{enumerate}
The definition of a martingale in \eqref{TwlwwLBPiK}  as a deterministic function
relates to its probabilistic interpretation in the same way that a random variable can be seen as
a deterministic function from a probability space to $\mathbb{R}$.
If we view martingales $M$ as deterministic functions satisfying
\eqref{TwlwwLBPiK}, and if we {\em require them to be non-negative}, 
then they provide a formalisation of a betting strategy  
 on an infinite coin-tossing game, where 
 $M(\sigma)$ denotes the capital at position $\sigma$.
Non-negativity expresses the requirement that
the player cannot borrow money after a bankruptcy, \ie upon the loss of 
all the capital, the game ends.
 Informally a bet consists of the favorable outcome (0 or 1) and the
 {\em wager}, which is the amount that will be won or lost after the outcome is revealed. 
 For convenience, we combine both of these parameters into the definition of the wager, whose sign
 reveals the favorable outcome:
\begin{equation}\label{M6wGjfXcnq}
w_M(\sigma):=M(\sigma\ast 1)-M(\sigma)
\hspace{0.5cm}\textrm{is the {\em wager at state} $\sigma$.}
\end{equation}
Hence if $w_M(\sigma)>0$ then the favorable outcome in this bet is 1; if
$w_M(\sigma)<0$ then the favorable outcome is 0. 
If $w_M(\sigma)=0$ then no bet is placed at position $\sigma$.
Wagers are usually called {\em martingale differences} in probability texts.
We say that $M$ {\em succeeds} on $X$ if 
\begin{equation}\label{64BkaTt1lE}
\limsup_n M(X\restr_n)=\infty.
\end{equation}
In order to consider realistic strategies it is natural to require that the martingales 
are definable or have some effectivity properties, for example that they are 
{\em computable} or {\em enumerable} by  a Turing machine. 
\begin{defi}[Computably enumerability of martingales]\label{pG43SuFmn}
A martingale  $M:\twomel\to\mathbb{R}^+$ is called {\em l.c.e.} if
$M(\sigma)$  can be approximated by an increasing computable sequence of rationals,
uniformly in $\sigma$.
Moreover we say that $M$ is {\em strongly l.c.e.} if is it \lce and the wagers $w_M(\sigma)$
 can be approximated by strictly monotone computable sequence of rationals,
uniformly in $\sigma$.\footnote{The reader may verify the following redundancy in the second clause of the definition: 
if the wagers $w_M(\sigma)$
 can be approximated by an increasing computable sequence of rationals,
uniformly in $\sigma$ and the initial capital $M(\lambda)$, where $\lambda$ is the empty string, 
is \lce (\ie has a computable increasing rational approximation)
then necessarily $M$  is a \lce martingale.}
\end{defi}
Computable and \lce martingales can be used as a foundation of algorithmic information theory,
see \citep[\S 13.2]{rodenisbook},  \citep{Li.Vitanyi:93} or \citep{martin_hist_rand}. A binary sequence to be
algorithmically random if no  \lce martingale $M$ {\em succeeds} on it in the sense of \eqref{64BkaTt1lE}.

{\bf Martingales and algorithmic randomness.}
It turns out that any \lce martingale $M$ can be transformed into a \lce martingale $N$ such that 
$\lim_n N(X\restr_n)=\infty$ for each $X$ such that \eqref{64BkaTt1lE} holds.
The betting strategies (or {\em unpredictability}) approach to algorithmic randomness 
is equivalent to the other two traditional approaches, 
namely the {\em incompressibility} approach (through Kolmogorov complexity)
and the {\em measure-theoretic} approach (through statistical tests). So a  real $X$ is \ml random 
(\ie roughly speaking, avoids all effective null sets) 
if and only if there exists some constant $c$ for which  $\forall n\ K(X\restr_n)>n-c$, where $K$ denotes the
\pf Kolmogorov complexity of $X$, if and only if no \lce (super)martingale succeeds on $X$.
The equivalence of the martingale approach with the other two, established in \cite{Schnorr:71}, 
is based on the {\em Kolmogorov inequality} (sometimes known 
as Ville's inequality as it appears in \cite{villemartin}) 
which  will be used in  \S \ref{aRvZodD8FN}, \S\ref{BBcmfYbEE} and says that
if $M$ is a martingale then:
\begin{equation}\label{xPWvKAIctR}
\sum_{\sigma\in S} 2^{-|\sigma|}\cdot M(\sigma)\leq M(\lambda)
\hspace{0.5cm}
\textrm{for each \pf set of strings $S$}
\end{equation}
where $\lambda$ denotes the empty string. 
If $S$ covers the whole space then equality holds, giving a version of the familiar fairness condition described by the martingale property.

\subsection{Monotonous strategies as martingales}\label{JbvJlths9B}
We formally define strategies that bet in a monotonous fashion, in terms of martingales.
\begin{defi}[Single-sided strategies]\label{VPTxEMxdFT}
A martingale $M$ is 0-sided if $M(\sigma\ast 0)\geq M(\sigma\ast 1)$ for all $\sigma$;
it is 1-sided if $M(\sigma\ast 1)\geq M(\sigma\ast 0)$ for all $\sigma$. We say that $M$ is
single-sided if it is either 0-sided or 1-sided. 
We say that $M$ is strictly single-sided if it is single sided and 
$M(\sigma\ast 0)\neq M(\sigma\ast 1)$ for all $\sigma$.
\end{defi}
%
%
%
A {\em prediction function} $f$ is a function from $\twomel$ to $\{0,1\}$. We say that $i<|\sigma|$
is a correct $f$-guess with respect to $\sigma$ if $f(\sigma\restr_i)=\sigma(i)$; otherwise we say that
$i$ is a false $f$-guess with respect to $\sigma$.
According to the two components (a), (b) of a betting strategy discussed in
\S\ref{VYuXf7zTQg}, a prediction function
can be seen as (b). 
%
\begin{defi}[Decidably-sided strategies]\label{C7kLPhxcBb}
Given a prediction function $f$,
a martingale $M$ is 
$f$-sided if any bias on the outcomes is decided by $f$; formally, if 
for all $\sigma$, $i$ if $M(\sigma\ast i)> M(\sigma)$
then $f(\sigma)=i$,  and similarly if 
$M(\sigma\ast i)< M(\sigma)$
then $f(\sigma)=1-i$.
A martingale $M$
is decidably-sided if
its favorable outcome is decidable, in the sense that
it is $f$-sided for a (total) computable prediction function $f$.
\end{defi}
Decidably-sided strategies can be seen as single-sided 
betting strategies modulo some effective re-naming of 0s and 1s. 
Another restricted strategy that we discussed
informally in \eqref{ZrEv5hZl2} is when the bets on 0s and the bets on 1s are based on separated
capital pools, with any winnings being returned to them, and losses taken from them, in a disjoint fashion.
These strategies are modeled by {\em separable martingales} which are martingales
that can be written as the sum of a 0-sided and a 1-sided martingale.

{\bf Facts and non-facts about monotonous betting.}
It is clear that $f$-sided and separable martingales are closed under 
(countable, subject to convergence of initial capitals) addition 
and multiplication by a constant.
Many of the facts about \lce martingales in the
beginning of \S\ref{a9WlWLzsyd} also hold
for the restricted martingales introduced above, by similar proofs.
Assuming that $f$ is computable:
\begin{equation*}
\parbox{14cm}{if $M$ is an $f$-sided martingale then there is an $f$-sided martingale $N$ with
$\lim_n N(X\restr_n)=\infty$ for all $X$ on which $M$ succeeds, in the sense of \eqref{64BkaTt1lE}. 
The same holds even if we replace `$f$-sided' with `separable' or `decidably-sided', 
or  qualify $M,N$ as \lce or computable.}
\end{equation*}
The proof is a simple adaptation of the standard argument, the so-called {\em savings trick}, 
(see  \citep[Proposition 6.3.8]{rodenisbook}).
Since algorithmic randomness can be defined with respect to a class of
effective (super)martingales, 
each of the restricted martingale notions that
we have discussed, \lce or computable, corresponds to a randomness notion.
Separating these notions is often a matter of adapting existing methods on this topic,
such as \citep[Chapter 7]{Ottobook}.
\begin{thm}[Partial computable strategies vs single-sided \lce strategies]\label{4N1PZvDiv}
There exists $X$ such that a 0-sided \lce martingale succeeds on $X$ and no
partial computable (super)martingale succeeds on $X$.
\end{thm}
The proof of Theorem \ref{4N1PZvDiv} is a straightforward adaptation of the arguments in 
\citep[\S 7.4]{Ottobook} and is thus left to the reader as an exercise. Our results 
in \S \ref{kB1LksuUZG} and \S \ref{BBcmfYbEE} 
can also be viewed as separations of randomness notions,
but their proofs require a novel argument. 
On the other hand, certain caution is needed as some basic facts
about (super)martingales and their effective versions, no-longer hold in the presence of
monotonousness. It is crucial to observe that {\em the difference of two single-sided martingales is not
always single-sided}, even if it is positive and even if they both favor the same outcome.
This is the reason why the two notions (i),(ii) of computably enumerable monotonous strategies
discussed in \S\ref{VYuXf7zTQg} are quite different. Another issue is that
under monotonousness, supermartingales are not interchangeable with martingales.
Classically, every supermartingale is bounded above by a martingale, and this is also true
for computable and \lce supermartingales (the \lce case is not straightforward; see 
\citep[\S 6.3]{rodenisbook}). Although this fact is also true
for single-sided supermartingales in the non-effective and computable cases, it can be shown to 
fail for \lce single-sided supermartingales.

\subsection{Mixtures of monotonous strategies, enumerability and approximations}\label{6E1fkvgPTH}
By the {\em mixture} of a finite or countable family $(M_i)$ of non-negative martingales we mean the sum
$M=\sum_i M_i$. In this terminology, there are two implicit assumptions:
 (a) the sum is bounded, in the sense that the total initial capital of the $M_i$ is finite:
 $\sum_i M_i(\lambda)<\infty$; (b) since we typically deal with effective or constructive strategies, 
we assume that $(M_i)$ has the same complexity, for example it is uniformly computable.
Mixtures of computable families of martingales allow for more powerful betting strategies since,
although $(M_i)$ is uniformly computable, the values of the 
capital $M=\sum_i M_i$  can only be approximated by a computable increasing sequence, uniformly in
the argument. Martingales $M$ with the latter approximation property are  
 \lce according to Definition \ref{pG43SuFmn} and are conceptually interesting since, 
 although the current capital $M(\sigma)$ and wager $w_M(\sigma)$ are measurable, \ie determined,
from the state $\sigma$, a constructive (computable) observer only has access to a approximations
of them. Hence even the favorable outcome may not be computable, while for strongly \lce martingales  
a computable observer has access to 
the favorable outcome as well as
a lower bound converging to the absolute value of the current wager.

{\bf Mixtures, enumerable strategies and optimality.}
The mixture of a computable family of martingales is a \lce martingale.
Moreover the mixture of a computable family of $f$-sided martingales is an
$f$-sided strongly \lce martingale.
In the following, we point out that the converse of these facts is true: 
every \lce martingale can be written as
the mixture of a computable family of martingales;
similarly, every strongly \lce strictly $f$-sided \lce martingale can be written as
the mixture of a computable family of strictly $f$-sided martingales.
These facts provide useful approximations for \lce monotonous martingales,  which 
will be used in \S\ref{aRvZodD8FN}, \S\ref{BBcmfYbEE}.
The reason that such respresentations are needed in the proofs that involve diagonalization,
is the somewhat surprising lack of universality in the class of \lce martingales.
By \citep{Downey.Griffiths.ea:04} there exists no effective enumeration of all \lce martingales.
This is usually an inconvenience in arguments which involve diagonalisation  against all
\lce martingales, and  a reason why
it is  often convenient to work with 
supermartingales (recall the discussion in \S\ref{JbvJlths9B} 
that effective martingales and supermartingales are exchangeable).
Since there exists a uniform enumeration 
of all \lce supermartingales, 
there exists a \lce supermartingale $M$ which is {\em optimal}, 
in the sense that any other \lce supermartingale is $\bigo{M}$, \ie multiplicatively dominated by $M$.
On the other hand, by \citep{Downey.Griffiths.ea:04} there is no optimal left-c.e.\  martingale $M$, \ie such that
any other left-c.e.\ martingale is $\bigo{M}$.
Unfortunately, {\em our arguments are specific to martingales} and do not apply to supermartingales.
This, along with the fact discussed in  
so we need to deal with the fact that, as discussed in  \S\ref{JbvJlths9B}, 
supermartingales are not exchangeable with martingales under monotonousness, 
means that we cannot use universality in our arguments.
\begin{lem}[Left-c.e.\ martingales as effective mixtures]\label{8tlYkfBkSB}
A martingale is \lce if and only if it can be written as the sum of a uniformly computable
sequence of  martingales.
\end{lem}
\begin{proof}
If $(N_i)$ is a uniformly computable sequence of martingales 
and $\sum_i N_i(\lambda)<\infty$ then it is well-known that
$\sigma\mapsto\sum_i N_i(\sigma)$ is a \lce martingale. For the converse, assume that 
$M$ is a \lce martingale and let $(M_s)$ be a \lce approximation to it so that
$M_{s+1}(\sigma)>M_{s}(\sigma)$ for all $s,\sigma$.
We define a family $(N_i)$ of martingales as follows.
Inductively assume that $N_i, i<k$ have been defined, they are martingales, and
\begin{equation}\label{YLFqQkQXED}
S_k(\sigma)< M(\sigma) 
\hspace{0.5cm} \textrm{for all $\sigma$, where $S_k:=\sum_{i<k} N_i$.}
\end{equation}
Consider a stage $s_0$ such that $M_{s_0}(\lambda)>\sum_{i<k} N_i(\lambda)$
and let $N_k(\lambda)=M_{s_0}(\lambda)-S_k (\lambda)$.
Then for each $\sigma$ suppose inductively that we have defined $N_k(\sigma)$ in 
such a way that $N_k(\sigma) +S_k (\sigma)\leq M_{t}(\sigma)$ for some
stage $t$. Since $M$ is a martingale,
this means that there exists some larger stage $s$ such that:
\begin{equation}\label{Yxju2tReUU}
M_s(\sigma\ast 0)+M_s(\sigma\ast 1)\geq
2N_k(\sigma) +2S_k (\sigma)=
2N_k(\sigma) + (S_k(\sigma\ast 0)+S_k(\sigma\ast 1)).
\end{equation}
Then we let $N_k(\sigma\ast i), i=\{0,1\}$ be two non-negative rationals such that:
\begin{enumerate}[\hspace{0.5cm}(a)]
\item $N_k(\sigma\ast 0)+N_k(\sigma\ast 1)=2N_k(\sigma)$;
\item $N_k(\sigma\ast i)+S_k(\sigma\ast i)\leq M_s(\sigma\ast i)$ for each $ i=\{0,1\}$.
\end{enumerate}
This concludes the inductive definition of $N_k$ 
and also verifies the property \eqref{YLFqQkQXED} for $k+1$ in place of $k$.
Note that the totality of each $N_i$ is guaranteed by the fact that $M$ is a martingale.
It remains to show that
\begin{equation}\label{TwatRWQaTU}
\lim_k S_k(\sigma)= M(\sigma)
\hspace{0.5cm}\textrm{for each $\sigma$.}
\end{equation}
By the definition of $N_i(\lambda)$, it follows that 
\eqref{TwatRWQaTU} holds for $\sigma=\lambda$. Assuming \eqref{TwatRWQaTU} for $\sigma$,
we show that it holds for $\sigma\ast i, i\in \{0,1\}$.
We have
\begin{equation}\label{YRM8JDJ42h}
M(\sigma\ast 0)+M(\sigma\ast 1)-S_k(\sigma\ast 0)-S_k(\sigma\ast 1)=
2M(\sigma)-2S_k(\sigma)=2(M(\sigma)-S_k(\sigma)),
\end{equation}
so by \eqref{TwatRWQaTU} we have:
$\lim_k S_k(\sigma\ast 0) + \lim_k S_k(\sigma\ast 1)=
M(\sigma\ast 0)+M(\sigma\ast 1)$.
By \eqref{YLFqQkQXED} applied to $\sigma\ast 0$ and $\sigma\ast 1$
we get $\lim_k S_k(\sigma\ast i)=M(\sigma\ast i)$ for $i\in\{0,1\}$, as required.
This concludes the inductive proof of \eqref{TwatRWQaTU}.
\end{proof}
%
{\bf Mixtures, monotonous betting and intermediate bets.}
Recall the two ways (i), (ii) that monotonous betting can be
considered for mixtures of strategies. We will show that for
mixtures of computable families of monotonous strategies,
these two formulations are essentially equivalent to
the two notions of computable enumerability of martingales in Definition  \ref{pG43SuFmn}.
The difference between (i) and (ii) is clear if we view
a mixture $S$ at a state $\sigma$ 
as  an infinite countable stack of bets 
that are being placed on the initial segments of $\sigma$. 
The crucial property of effective single-sided strategies $S$ under (i), is that they are
effectively approximated by single-sided strategies $(S_i)$ such that for each $n<m$,
{\em the intermediate bets $S_m-S_n$ are also single-sided}. Since in general the
difference of single-sided strategies may not be single-sided, this property
may not be present under clause (ii).
A computable observer
can only access a certain approximation to $S$ at each stage, \ie a certain finite initial segment of
the bets that compose $S$. At later stages the observer has access a more accurate approximation: 
{\em the intermediate bets express the error of the first observation with respect to the current one}.
For an analogue of Lemma \ref{8tlYkfBkSB} in the case
of monotonous \lce martingales (non-uniform case (ii)) 
we require 
strict monotonousness in the sense of Definitions \ref{VPTxEMxdFT} and 
\ref{C7kLPhxcBb}, \ie 
that a non-empty bet is placed at every state.
This requirement is not essential, as Lemma \ref{XXbFLuIodX} shows.
\begin{lem}\label{XXbFLuIodX}
If $f$ is a computable prediction function, then
for each \lce $f$-sided martingale $M$ we can effectively obtain a
 \lce strictly $f$-sided martingale $\hat{M}$ such that for each $X$ with 
$\limsup_s M(X\restr_s)=\infty$ we have $\limsup_s \hat{M}(X\restr_s)=\infty$.
Hence if no strictly $f$-sided \lce martingale succeeds on a real $Y$, then
no $f$-sided \lce martingale succeeds on $Y$.
\end{lem}
\begin{proof}
Let $N$ be the computable martingale which starts with $N(\lambda)=1$ and
at each $\sigma$, it bets half of $N(\sigma)$ on $f(\sigma)$. Define
$\hat{M}=M+N$ so $\hat{M}$ is clearly $f$-sided and succeeds on every real that $M$ does.
Since $N(\sigma)>0$ 
for all $\sigma$ it follows that $N$ is strictly $f$-sided. Then 
$\hat{M}(\sigma\ast f(\sigma))-\hat{M}(\sigma\ast (1-f(\sigma)))$ equals
\[
(M(\sigma\ast f(\sigma))-M(\sigma\ast (1-f(\sigma)))+
(N(\sigma\ast f(\sigma))-N(\sigma\ast (1-f(\sigma)))
\]
which is larger than 0 as required, since $M$ is $f$-sided and $N$ is strictly $f$-sided.
\end{proof}
\begin{lem}[Monotonous left-c.e.\ martingales as mixtures]\label{XSQ74P8ao}
For every computable prediction function $f$ and every \lce strictly
$f$-sided martingale $M$, 
there exists a uniformly computable sequence $(N_i)$ of martingales such that 
the partial sums $S_n=\sum_{i<n} N_i$ are $f$-sided  and
converge to $M$. 
\end{lem}
\begin{proof}
The proof is a simple adaptation of the proof of Lemma \ref{8tlYkfBkSB}, so we may refer
to the displayed equations in that proof, although the parameters have a modified meaning that
we determine below.
We give the proof of the case of single-sided martingales, as the case of decidably-sided
martingales is entirely analogous.
Without loss of generality, assume that $M$ is a \lce and 0-sided martingale.
and let $(M_s)$ be a \lce approximation to it so that
$M_{s+1}(\sigma)>M_{s}(\sigma)$ for all $s,\sigma$.
We define a computable family $(N_i)$ of martingales:
inductively assume that $N_i, i<k$ have been defined and are martingales, and
\begin{equation}\label{YLFqQkQXEDa}
S_k:=\sum_{i<k} N_i
\hspace{0.3cm} \textrm{is 0-sided, and for all $\sigma$, \hspace{0.3cm}}
S_k(\sigma)< M(\sigma) 
\end{equation}
Consider a stage $s_0$ such that $M_{s_0}(\lambda)>\sum_{i<k} N_i(\lambda)$
and let $N_k(\lambda)=M_{s_0}(\lambda)-S_k (\lambda)$.
Given $\sigma$, suppose inductively that 
we have defined $N_k(\sigma)$ in such a way that $N_k(\sigma) +S_k (\sigma)\leq M_{t}(\sigma)$ for some
stage $t$, and for each $\rho\prec\sigma$ we have $S_{k+1}(\rho\ast 0)\geq S_{k+1}(\rho\ast 1)$.
Since $M$ is a 0-sided martingale,
there exists some $s>t$ such that \eqref{Yxju2tReUU}
%
%
and $M_s(\sigma\ast 0)> M_s(\sigma\ast 1)$.
Then we let $N_k(\sigma\ast i), i=\{0,1\}$ be two non-negative rationals such that:
\begin{enumerate}[\hspace{0.5cm}(a)]
\item $N_k(\sigma\ast 0)+N_k(\sigma\ast 1)=2N_k(\sigma)$;
\item $N_k(\sigma\ast i)+S_k(\sigma\ast i)\leq M_s(\sigma\ast i)$ for each $ i=\{0,1\}$.
\item $N_k(\sigma\ast 1)+S_k(\sigma\ast 1)\leq N_k(\sigma\ast 0)+S_k(\sigma\ast 0)$ for each $ i=\{0,1\}$.
\end{enumerate}
This concludes the inductive definition of $N_k$ 
and also verifies the property \eqref{YLFqQkQXEDa} for $k+1$ in place of $k$.
\begin{equation}\label{cC9T8dYWba}
\parbox{14cm}{{\em Remark on the definition:} If 
$M_s(\sigma\ast 0)-M_s(\sigma\ast 1)<M_t(\sigma\ast 0)-M_t(\sigma\ast 1)$,
it is possible that the chosen values satisfy  $N_k(\sigma\ast 1)>N_k(\sigma\ast 0)$,
in which case $N_k$ is not 0-sided.}
\end{equation}
The totality of each $N_i$ is guaranteed by the fact that $M$ is a strictly 0-sided martingale.
It remains to show \eqref{TwatRWQaTU}.
%
%
From the definition of each $N_i(\lambda)$, it follows that 
\eqref{TwatRWQaTU} holds for $\sigma=\lambda$. Assuming \eqref{TwatRWQaTU} for $\sigma$,
we show that it holds for $\sigma\ast i, i\in \{0,1\}$.
We have \eqref{YRM8JDJ42h} as before,
so by \eqref{TwatRWQaTU} we have:
$\lim_k S_k(\sigma\ast 0) + \lim_k S_k(\sigma\ast 1)=
M(\sigma\ast 0)+M(\sigma\ast 1)$.
By \eqref{YLFqQkQXEDa} applied to $\sigma\ast 0$ and $\sigma\ast 1$
we get $\lim_k S_k(\sigma\ast i)=M(\sigma\ast i)$ for $i\in\{0,1\}$, as required.
This concludes the induction for \eqref{TwatRWQaTU}.
\end{proof}
\begin{lem}[Monotonous strongly left-c.e.\ martingales as mixtures]\label{iQMAP9apXx}
For every computable prediction function $f$ and every strictly
$f$-sided strongly \lce martingale $M$, 
there exists a uniformly computable sequence $(N_i)$ of $f$-sided martingales such that 
the partial sums $S_n=\sum_{i<n} N_i$ 
converge to $M$. 
\end{lem}
\begin{proof}
We do the proof for the case when $M$ is strictly 0-sided, as the more general case
is entirely similar.
By \eqref{cC9T8dYWba} the application of 
the construction in Lemma \ref{XSQ74P8ao} to the given $M$ does not ensure that 
the $N_i$ are 0-sided. In order to achieve this, we note that since $M$ is
assumed to be strongly \lce and 0-sided, there exists a computable \lce approximation 
$(M_s)$ to $M$ such that $M_{t}(\sigma)>M_{s}(\sigma)$ and
$M_s(\sigma\ast 0)-M_s(\sigma\ast 1)<M_t(\sigma\ast 0)-M_t(\sigma\ast 1)$
for all $s<t$ and all $\sigma$. Using this approximation $(M_s)$ we can apply the construction
 in  Lemma \ref{XSQ74P8ao} with the extra clause
(d): $N_k(\sigma\ast 1)\leq N_k(\sigma\ast 0)$ in the
induction step for the definition of  $N_k(\sigma\ast i), i=\{0,1\}$.
This extra clause does not affect the existing argument, hence the constructed $(N_k)$
is a computable family of martingales such that their mixture converges to $M$. In addition,
clause (d) in the inductive definition directly guarantees that each 
$N_k$ are 0-sided.
\end{proof}
The constructions in
\S \ref{aRvZodD8FN}, \S\ref{BBcmfYbEE} rely on the existence of certain
`canonical' effective approximations.
\begin{defi}[Canonical approximations of monotonous martingales]\label{M1XF9CwRuD}
Given a computable prediction function $f$ and a \lce $f$-sided martingale $N$,
a {\em canonical approximation} to $N$ is a computable family $(S_i)$
of $f$-sided martingales that converge to $N$ such that each $S_{i+1}-S_i$
is also an $f$-sided martingale.
\end{defi}
By Lemma \ref{XSQ74P8ao} every strictly decidably-sided 
strongly \lce martingale has a canonical approximation.
Given single-sided martingales
$N,T$, we say that  $(M_s)$ is a canonical approximation to
the separable martingale
$M=N+T$ if $M_s=N_s+T_s$ for
canonical approximations $(N_s), (T_s)$ of $N,T$ respectively.

\subsection{Monotonous betting on a biased coin}\label{OC53m9ZSeG}
We give two examples of types of biases can be exploited
through single-sided or separable strategies,
establishing basic properties of monotonous betting that will be used mainly in \S\ref{Sp9icqFzvI}.

{\bf Monotonous betting on Villes' casino sequence.} 
A well-known\footnote{Short expositions of the debate in relation to the notion of 
algorithmic randomness can be found on textbooks on this topic such as
\cite[\S1.9]{Li.Vitanyi:93} and \cite[\S 6.2]{rodenisbook}. Extended discussions of the philosophical
underpinnings of this debate can be found in \cite{vanlamb87} and the more recent \cite{Blandothesis}.} 
debate in the early days of probability occurred between the 
competing approaches
of Kolmogorov, which won the debate, and the frequentist-based approach of von Mises, 
for the establishment of the foundations of probability. A significant factor for the loss of support
to von Mises' theory was a certain casino sequence constructed by \cite{villemartin}\footnote{An English
translation can be found at \texttt{http://www.probabilityandfinance.com/misc/ville1939.pdf}.
Simpler proofs of Ville's theorem appear in \cite{villeeasier} and \cite[\S 6.5]{rodenisbook}} which
is `random' with respect to any given countable collection of choice sequences (a basic tool in von Mises' 
strictly frequentist approach) but is biased according to a well-accepted statistical test: although the
frequency
of 0s approaches 1/2, in all initial segments this frequency never drops below 1/2.
We point out that the bias in Villes' well-known example is exploitable by
computable monotonous betting.
In order to see this, let
$z_n,o_n$ be the number of 0s and 1s respectively, 
in the first $n$ bits of Ville's casino sequence, so that
$z_n \geq o_n$ for all $n$.
In the case where
$\sup_n (z_n-o_n) = \infty$ our strategy is to 
start with capital 1, and bet wager 1 on outcome 0 at each step.
In the case where
$\limsup_n (z_n-o_n):=k<\infty$, given $k$ and a stage $t$
such that for all $n\geq t$ we have $z_n-o_n\leq k$, we can used the following strategy:
given any stage $s_0>t$, find some $n\geq s_0$ such that 
$z_n-o_n=k$ and at this $n$ bet on 1. In order to avoid the dependence of this strategy on
the parameters $k,t$, we can consider a mixture including a strategy for 
each possible pair $(k,t)$, with initial capital for the $s$-th strategy equal to $2^{-s}$
(so that the total initial capital is finite).Note that in the first case the strategy is 0-sided
and in the second case it is 1-sided; moreover in both cases, under the respective assumption,
the strategies are successful on Ville's casino sequence.
The mixture of these two strategies is a computable 
separable strategy and is successful on 
Ville's sequence.

{\bf Monotonous betting for skewed or non-existent limiting frequency.} 
Given a casino sequence $X$ with limiting frequency of 0s different than 1/2,
there is a single-sided betting strategy that is successful on $X$. Moreover
there is a separable martingale which succeeds on every such $X$, 
irrespective of whether the frequency is above or below 1/2, or even how much
it differs from 1/2. A slightly more general version of these facts,
is the following form of
Hoeffding's inequality which we prove via of betting strategies,
and which will be used in our later arguments.. 
\begin{lem}[Hoeffding for prediction functions]\label{PvVOtknZpc}
Given  $q>1/2$, $n\in \omega$ and a prediction function $f$, 
the number of strings in $2^n$ for which the number of correct $f$-guesses is 
more than $qn$ is at most $r_q^{-n}\cdot 2^{n}$, where $r_q>1$ is a function of $q$.
So the number of strings in $2^n$ for which the number of correct $f$-guesses is in 
the interval $((1-q)n, qn)$ is at least $2^n\cdot (1-2 r_q^{-n})$.
\end{lem}
\begin{proof}
Given $f$, let 
$z_{\sigma}$ denote the number of correct $f$-guesses with respect to $\sigma$, and let $o_{\sigma}$
be the number of false $f$-guesses with respect to  $\sigma$.
For each $q>1/2$, consider the function $d:\twomel\to\mathbb{R}^{+}$ 
defined by $d(\lambda)=1$ and 
$d(\sigma)=2^{|\sigma|}\cdot q^{z_{\sigma}}\cdot (1-q)^{o_{\sigma}}$.
Note that, if $f(\sigma)=0$ then:
\[
d(\sigma\ast 0)+d(\sigma\ast 1)=2^{|\sigma|+1}\cdot \left(q^{z_{\sigma}+1}\cdot (1-q)^{o_{\sigma}}+
q^{z_{\sigma}}\cdot (1-q)^{o_{\sigma}+1}\right)=
2^{|\sigma|+1}\cdot q^{z_{\sigma}}\cdot (1-q)^{o_{\sigma}}=2d(\sigma).
\]
The same is true in the case that $f(\sigma)=1$, so that $d$
is a martingale, which bets 
$|d(\sigma\ast 0)-d(\sigma)|=(2q-1)d(\sigma)$
on the prediction of $f$ at $\sigma$.
For each $\sigma$ let $p_{\sigma}=z_{\sigma}/|\sigma|$, 
so that $1-p_{\sigma}=o_{\sigma}/|\sigma|$. Suppose that $p_{\sigma}>q$. Then
\[
d(\sigma)= \left(2\cdot q^{p_{\sigma}}\cdot (1-q)^{1-p_{\sigma}}\right)^{|\sigma|}>
\left(2\cdot q^{q}\cdot (1-q)^{1-q}\right)^{|\sigma|}
\]
where the second inequality holds because the function $x\mapsto 2q^x(1-q)^{1-x}$
is increasing\footnote{The derivative of $x\mapsto q^x(1-q)^{1-x}$ 
is $\big(\log q-\log (1-q)\big)\cdot (1-q)^{1-x}\cdot  q^x$.} in $(0,1)$ when $q>1/2$.
Again by considering the derivatives, we can see that the function $q\mapsto q^q\cdot (1-q)^{1-q}$ is
decreasing in $(0,1/2)$, increasing in $(1/2,1)$ and it has a global minimum in $(0,1)$ at $q=1/2$, 
at which point it takes the value 1/2. So if we let $r_q:=2q^q\cdot (1-q)^{1-q}$ and recall that $q>1/2$ we get
$r_q>1$ and $d(\sigma)\geq r_q^{|\sigma|}$ for each $\sigma$ with $p_{\sigma}>q$. 
From  Kolmogorov's inequality in then follows that, if $t_n$ is the number of strings $\sigma\in 2^n$
with $p_{\sigma}>q$, then
$t_n\cdot 2^{-n}<r_q^{-n}$. So $t_n<r_q^{-n}\cdot 2^{n}$ as required.
\end{proof}
%

{\bf Computable single-sided randomness and frequency.}
Lemma \ref{PvVOtknZpc} 
says that for each total prediction function $f$, 
with high probability the number of correct $f$-guesses along a 
binary string $\sigma$ are concentrated around $|\sigma|/2$.
In fact, there exists
a separable computable martingale which succeeds on every stream $X$ with the property that
the proportion of correct $f$-guesses along $X$ does not reach limit 1/2.
For each $q\in (1/2,1)$ let $T_q(\lambda)=1$,
and define $T_q(\sigma)=2^{|\sigma|}\cdot q^{z_{\sigma}}\cdot (1-q)^{o_{\sigma}}$
where $z_{\sigma}$ is the number of correct $f$-guesses with respect to $\sigma$ and 
$o_{\sigma}$ is the number of false $f$-guesses with respect to $\sigma$.
By the proof of Lemma \ref{PvVOtknZpc}, $T_q(\sigma)$ is a martingale
and $\limsup_s T_q(X\restr_n)=\infty$ for each $X$ such that 
$\limsup_s z_{X\restr_n}/n>q$. Similarly, 
$T_{q}(\sigma)$ is a martingale for each $q<1/2$,
and $\limsup_s T_q(X\restr_n)=\infty$ for each $X$ such that 
$\limsup_s z_{X\restr_n}/n<q$.
Let $q_i=1/2+2^{-i-1}$ and $p_i=1/2-2^{-i-1}$ for each $i$ and define:
\[
N(\sigma)=\sum_{i} 2^{-i} \cdot T_{q_i}(\sigma)+
\sum_{i} 2^{-i} \cdot T_{p_i}(\sigma).
\]
Then $N$ is a computable martingale and by the properties of
$T_{q_i}, T_{p_i}$, it succeeds on every $X$ for which the
proportion of correct $f$-guesses does not tend to 1/2. In the case that $f$
is the constant zero function  $T_q$ is 0-sided, which implies the following fact, where
`computably single-sided random' is a sequence where no computable single-sided (super)martingale succeeds.
\begin{equation}\label{GCBYAq98IA}
\parbox{14cm}{There exist computable families $(N_i)$, $(T_i)$ of single-sided strategies 
such that $\sum_i (N_i+T_i)$ has finite initial capital and succeeds on all 
sequences whose limiting 0-frequency is not 1/2. Hence
each computably single-sided random has  0-frequency tending to 1/2.}
\end{equation}
By `0-frequency' we mean the (relative) frequency of 0 in the initial segments of the sequence.
Hence weak s-randomness for $s\in (0,1)$ does not imply computable single-sided randomness. 
However 
in Proposition \ref{LXn23yy81K}, 
we will see that
the \lce version of single-sided randomness does imply weak $1/2$-randomness.

\subsection{Speed of success and effective Hausdorff dimension}\label{Sp9icqFzvI}
The martingale approach to algorithmic information theory was introduced by \citep{Schnorr:71, Schnorr:75}
who also showed some interest in the rate of success of (super)martingales $M$, and in particular the classes
\[
S_h(M)=\left\{X\ |\ \limsup_n \frac{M(X\restr_n)}{h(n)}=\infty\right\}
\]
where $h:\Nat\to\Nat$ is a computable non-decreasing function.
Later \citep{Lutz:00,Lutz:03} showed that the Hausdorff dimension of a class of reals can be characterized
by the exponential `success rates' of \lce supermartingales, and in that light defined the 
effective Hausdorff dimension $\dim (X) $ of a real $X$ as the infimum of the  $s\in (0,1)$ such that $X\in S_{h}(M)$ for some
\lce supermartingale $M$, where $h(n)=2^{(1-s)n}$.
Then \citep{Mayordomo:02} showed that
\begin{equation}\label{kleTr2ZET}
\dim (X)=\text{lim inf}_n \frac{C(X\restr_n)}{n}=\text{lim inf}_n \frac{K(X\restr_n)}{n}
\end{equation}
where $C$ and $K$ denote the plain and \pf Kolmogorov complexity respectively.
Reals with effective Hausdorff dimension 
1/2 include partially predictable reals (with an imbalance of 0s and 1s) like
$Y\oplus \emptyset$ where $Y$ is algorithmically random, as well as 
random-looking versions of the halting probability like
$\sum_{U(\sigma)\de} 2^{-2|\sigma|}$ for certain universal \pf machines $U$ from 
\citep{MR1888278}. \ml random reals have effective dimension 1, 
but the converse does not hold. Moreover there
are computably random reals of effective dimension 0. 
For more on algorithmic dimension see
\citep[Chapter 13]{rodenisbook}.

{\bf Monotonous betting on sequences with dimension less than half.}
We construct a computable mixture of separable strategies,
which succeeds on every sequence of effective Hausdorff dimension $<1/2$.
For this task we need a characterization of effective dimension in terms of tests.
Given $s\in (0,1)$, an $s$-test is a uniformly \ce sequence $(V_i)$ of sets of strings such that
$\sum_{\sigma\in V_k} 2^{-s|\sigma|}<2^{-k}$ for each $k$. 
As mentioned in \cite[\S 13.6]{rodenisbook} and is the 
case for most notions of effective statistical tests, 
\begin{equation}\label{IWvHgcPRHK}
\parbox{11cm}{given $s\in (0,1)$ one can effectively obtain an
effective list of all $s$-tests.}
\end{equation}
Since $s<1$, the condition $\sum_{\sigma\in V_k} 2^{-s|\sigma|}<2^{-k}$ means that
the length of each string in $V_k$ is more than $k$.
These observations will be used in the proof of Theorem \ref{LXn23yy81K}.
Let us say that $X$ is
{\em weakly $s$-random} if it avoids all $s$-tests $(V_i)$, in the sense that there are only finitely many 
$i$ such that $X$ has a prefix in $V_i$. By
\citep{MR1888278}, $X$ being weakly $s$-random is equivalent to   
$\exists c\ \forall n\ K(X\restr_n)>s\cdot n-c$. Then by \eqref{kleTr2ZET},
\begin{equation}\label{DitiNX8zan}
\dim (X)=\sup \{s\ |\ \textrm{$X$ is weakly $s$-random}\}
\end{equation}
which is crucial for the proof of the following fact, which
complements our main theorems in 
\S \ref{kB1LksuUZG} and \S \ref{BBcmfYbEE}.
\begin{thm}[Monotonous betting for low dimension]\label{LXn23yy81K}
There exist uniformly computable 0-sided and 1-sided strategies $(N_i)$ and $(T_i)$ respectively
such that the mixture $\sum_i (N_i+T_i)$ has finite initial capital and succeeds on all $X$ such that
$\dim (X) <1/2$.
\end{thm}
\begin{proof}
By \eqref{GCBYAq98IA} that we established in 
\S\ref{OC53m9ZSeG}, it suffices to construct a computable
family $(N_i)$ of 0-sided strategies such that $\sum_i N_i$ has finite initial capital
and succeeds on every sequence $X$ which has
limiting 0-frequency 1/2 and
 $\dim (X) <1/2$. For each  $X$ with these properties, by \eqref{DitiNX8zan}
there exists a rational $q<1/2$ and a $q$-test $(V_i)$
such that $X$ has prefixes in infinitely many $V_i$.
It suffices to prove that:
\begin{equation}\label{YKrvZ2mzRq}
\parbox{14cm}{given $\epsilon>0$, $q<1/2$ and a $q$-test $(V_i)$, we can effectively
define a computable family $(M_{i})$ of 0-sided strategies such that
$\sum_i M_i$ has initial capital less than $\epsilon$ and succeeds on every $X$ with limiting 0-frequency
equal to 1/2 and prefixes in infinitely many members of $(V_i)$.}
\end{equation}
Indeed, given \eqref{YKrvZ2mzRq}
and \eqref{IWvHgcPRHK} we can effectively produce a family of 0-sided strategies
whose mixture has bounded initial capital and deal with any possible 
$q$-test $(V_i)$ for any choice of $q<1/2$.
For the proof of  \eqref{YKrvZ2mzRq}, 
given $\epsilon>0$, $q<1/2$ and a $q$-test $(V_i)$, 
let $k_{\epsilon}$ be the least integer such that
$2^{-k_{\epsilon}}<\epsilon/2$.
We define a computable family $(N_{\sigma})$ of 0-sided strategies (indexed 
by strings) and let 
\[
M_i=\sum_{\sigma\in V_{k_{\epsilon} + i}} N_{\sigma}
\hspace{0.5cm}\textrm{and}\hspace{0.5cm}
M=\sum_i M_i.
\]
Under this definition of $M_i$ and the
properties $q$-tests, 
for $M(\lambda)<\epsilon$
it suffices to let $N_{\sigma}(\lambda)=2^{-q|\sigma|}$ 
so that
\[
M_i(\lambda)=
\sum_{\sigma\in V_{k_{\epsilon} + i}} N_{\sigma}(\lambda)=
\sum_{\sigma\in V_{k_{\epsilon} + i}} 2^{-q|\sigma|}< 2^{-k_{\epsilon} - i}
\Rightarrow 
M(\lambda)<2\cdot 2^{-k_{\epsilon}} < \epsilon.
\]
For each $i$ and each $\sigma$ strategy
$N_{\sigma}$ starts with 
$N_{\sigma}(\lambda)=2^{-q|\sigma|}$ and 
bets all capital on all the 0s of $\sigma$, while placing no bets on all other strings.  
Formally, for each non-empty $\rho$ define:
\[
N_{\sigma}(\rho)=
\left\{\begin{array}{cl}
N_{\sigma}(\hat{\rho})&\textrm{if $\hat{\rho}\ast 0\not\preceq\sigma$}\\
2\cdot N_{\sigma}(\hat{\rho})&\textrm{otherwise.}
\end{array}\right\}
\hspace{0.3cm}\textrm{where $\hat{\rho}$ denotes the predecessor of $\rho$.}
\]
Since each $N_{\sigma}$ is 0-sided, $M_i$ and $M$ 
are also 0-sided and, as noted above, $M(\lambda)<\epsilon$. Hence 
for \eqref{YKrvZ2mzRq} it remains to
verify that $M$ succeeds on every $X$ with limiting 0-frequency 1/2 and 
prefixes in infinitely many members of $(V_i)$.
To this end we observe that, as a direct consequence of 
the definitions of $N_{\sigma}, M_i, M$: 
\begin{equation}\label{lVon8b4uYY}
\parbox{13cm}{if $\sigma\in V_{k_{\epsilon}+i}$ 
has $z_{\sigma}$ many 0s then for each $\rho\succeq\sigma$, 
$M_i(\rho)\geq N_{\sigma}(\rho)=N_{\sigma}(\sigma) =
2^{z_{\sigma}-q|\sigma|}$.}
\end{equation}
Given $X$ with limiting 0-frequency 1/2 and 
prefixes in infinitely many members of $(V_i)$, let $q'$ be a rational in $(q,1/2)$.
Since the limiting frequency of 0s in $X$ is 1/2, 
there exists some $n_0$ such that for each $n>n_0$ the number $z_{X\restr_n}$ of 0s 
in $X\restr_n$ is more than $q'n$, so
$2^{z_{X\restr_n}-qn}>2^{(q'-q)n}$. Given any constant $c$, let $n_1>n_0$ be such that
$2^{(q'-q)n_1}>c$. Let $n_2>\max\{n_1,k_{\epsilon}\}$ be such that $X$ has a prefix in 
$V_{n_2}$ and all strings in $V_{n_2}$ are of length at least $n_1$. 
If $\sigma$ is a prefix of $X$ in $V_{n_2}$, by the definition
of $M$ and \eqref{lVon8b4uYY},  for all $n\geq n_2$  we have 
\[
M(X\restr_n)\geq M_{n_2}(X\restr_n)\geq 
N_{\sigma}(\sigma)\geq 2^{z_{\sigma}-q}>
2^{(q'-q)n_1}>c.
\]
Since $c$ was arbitrary, this shows that
that $\lim_n M(X\restr_n)=\infty$ for all $X$ with the properties of \eqref{YKrvZ2mzRq}.
\end{proof}

\section{The power of single-sided martingales and their mixtures}\label{kB1LksuUZG}\label{aRvZodD8FN}
We show that if a computable martingale succeeds on some casino sequence $X$,
then there exists a computable single-sided martingale which succeeds on $X$.
This is a consequence of the following decomposition,
which was also noticed independently by Frank Stephan.
\begin{lem}[Single-sided decomposition]\label{aycCdojlpe} 
Every martingale $M$ is the product of a 0-sided martingale $N$ and a 1-sided martingale $T$.
Moreover $N,T$ are computable from $M$.
\end{lem}
\begin{proof}
For ease of notation we let $M_{\sigma},N_{\sigma},T_{\sigma}$ denote
$M(\sigma),N(\sigma),T(\sigma)$ respectively.
Let $\sigma \to w_M(\sigma)$ be the wagers of $M$ and
let $N_{\lambda}=T_{\lambda}=\sqrt{M_{\lambda}}$.
Define the wagers of $N,T$ respectively by:
\[
w_N(\sigma)=
\left\{\begin{array}{cl}
w_M(\sigma)/ T_{\sigma}&\textrm{if $w_M(\sigma)<0$, $T_{\sigma}>0$;}\\
0&\textrm{otherwise;}
\end{array}\hspace{-0.1cm}\right\},
\hspace{0.3cm}
w_T(\sigma)=
\left\{\begin{array}{cl}
w_M(\sigma)/N_{\sigma}&\textrm{if $w_M(\sigma)>0$, $N_{\sigma}>0$;}\\
0&\textrm{otherwise;}
\end{array}\hspace{-0.1cm}\right\}
\]
We show by induction that $M_{\sigma}=N_{\sigma}\cdot T_{\sigma}$ for all $\sigma$.
By definition this holds for $\sigma=\lambda$; suppose that it holds for $\sigma$.
If $w_M(\sigma)=0$ then  $w_N(\sigma)=w_T(\sigma)=0$ so
 $M_{\rho}=N_{\rho}\cdot T_{\rho}$ holds for the immediate successors $\rho$
 of $\sigma$. If
$w_M(\sigma)<0$ then $w_T(\sigma)=0$ so $T_{\sigma\ast i}=T_{\sigma}$ and
\[
M_{\sigma\ast 1}=M_{\sigma}+w_M(\sigma)=N_{\sigma}\cdot T_{\sigma}+w_M(\sigma)=
T_{\sigma}\cdot (N_{\sigma}+w_N(\sigma))=T_{\sigma\ast 1}\cdot N_{\sigma\ast 1}.
\]
In the same way we have $M_{\sigma\ast 0}=T_{\sigma\ast 0}\cdot N_{\sigma\ast 0}$.
The case where $w_M(\sigma)>0$ is entirely symmetric.
\end{proof}
\begin{coro}\label{jmTLCiryr}
Given a computable martingale $M$, there exist a 0-sided martingale $N_0$ and a
1-sided martingale $N_1$ such that 
for each $X$ on which $M$ is successful, at least one of $N_0, N_1$ is successful.
\end{coro}
Corollary \ref{jmTLCiryr} 
is a direct consequence of Lemma \ref{aycCdojlpe}.
and
says that, in terms of computable strategies, if there
exists a successful strategy against the casino, there exists a successful single-sided
strategy. 
This fact is no longer true for mixtures or 
strongly \lce martingales (recall the equivalence from \S\ref{6E1fkvgPTH}).
\begin{thm}[Mixtures of single-sided martingales]\label{1iotxpCAHo}
There exists a real of effective Hausdorff dimension $1/2$
such that no single-sided (or separable) strongly \lce martingale succeeds on it.
\end{thm}
It is instructive to contrast Theorem \ref{1iotxpCAHo}
with Proposition \ref{LXn23yy81K}. Note that
for each rational $s\in (0,1)$ there are
reals $X$ with effective Hausdorff dimension $s$ with computable subsequences, 
so that single-sided strategies succeed
easily on them.
The basic idea for proving Theorem \ref{1iotxpCAHo} is most clearly
demonstrated by proving the following simpler statement, which only deals with 
a single  separable strategy:
\begin{equation}\label{hEsg1rmVV2}
\parbox{14cm}{Given the mixture $M$ of a computable family $(M_i)$ of 
separable martingales $M$, there exists a \lce 
real of effective Hausdorff dimension $1/2$
such that $M$ does not succeed on it.}
\end{equation}
The proof of \eqref{hEsg1rmVV2} is a computable construction of the approximation to the
required real $X$, and is presented in 
\S\ref{96WqHtN5Iv}--\S\ref{n6OP5Cpi6D} in a modular way, so that it can be used 
in the more involved  proof of Theorem \ref{1iotxpCAHo}.
The only issue that separates
the proof of \eqref{hEsg1rmVV2}
form the proof of Theorem \ref{1iotxpCAHo}
is the lack of universality and effective lists of martingales that was
discussed in \S\ref{6E1fkvgPTH}.

\subsection{Idea and plan for the proof of \eqref{hEsg1rmVV2}}\label{96WqHtN5Iv}
We will construct the real $X$ of \eqref{hEsg1rmVV2}
so as to extend a sequence  of initial segments $(\sigma_n)$. 
Given $M$ as in \eqref{hEsg1rmVV2}
we will construct $X$ of effective Hausdorff dimension 1/2 such that $M(X\restr_n)$ is bounded above.
Without loss of generality we can assume that $M(\lambda)<2^{-1}$.
In order to ensure the dimension requirement for $X$, it suffices to ensure that
\begin{equation}\label{ivX9i8OUls}
K_V(\sigma_n)\leq |\sigma_n|\cdot q_n
\end{equation}
for all $n$, where $V$ is a \pf machine that we also construct, $K_V$ is the 
Kolmogorov complexity with respect to $V$,  and $(q_n)$
is a computable decreasing sequence of rationals tending to 1/2. Let us set 
\begin{equation}\label{XOg6jVNTyF}
q_n=1/2+ 3/(n+2)
\hspace{0.5cm}\textrm{and}\hspace{0.5cm}
\hat{M}(\sigma)=\max_{n\leq |\sigma|} M(\sigma\restr_n)
\end{equation}
We will ensure that for all $n$:
\begin{equation}\label{FRPmaVyVWx}
\hat{M}(\sigma_n)\leq  2^{-1}+\sum_{i< n} 2^{-i-2}.
\end{equation}
One way to think about this requirement is to try to ensure that
$\hat{M}(\sigma_n)-\hat{M}(\sigma_{n-1})\leq 2^{-n-1}$ for all $n$.
Supposing inductively that $\sigma_{n-1}$ has been determined, the task of 
keeping $\hat{M}(\sigma_n)-\hat{M}(\sigma_{n-1})$ 
small potentially involves  changing the approximation to $\sigma_n$ a number of times,
since $M$ is a \lce martingale. 
This instability of the final value of $\sigma_n$ is in conflict with 
\eqref{ivX9i8OUls}.
The main idea for handling this conflict is that if we choose $\sigma_n$
from a collection of strings which have roughly similar number of 0s and 1s, then a single-sided
strategy is limited to winning on around half of the available bits.
With such a limitation on the components of $M$, the separability of $M$ ensures that
the growth potential of $M$ is also limited, in a way that allows the satisfaction of  \eqref{ivX9i8OUls}.
Since the construction deals with approximations $(N_i)$ of $M$, it is crucial for this argument that
the intermediate bets $N_t-N_s$ between two stages $s<t$ are single-sided, or separable.
As we discussed in \S\ref{JbvJlths9B},
such an approximation can be chosen when $M$ is a mixture of a computable family $(M_i)$
of separable strategies. 

The next concern, given the restriction to strings with balanced number of 0s and 1s,
is to be able to  choose an extension of $\sigma_n$ where 
capital does not increase substantially (note that without the restriction to
a particular set of extensions of $\sigma_n$,  we can
choose an extension where the capital does not increase at all). 
%
\begin{lem}[Low capital gain somewhere]\label{FDhethbip9}
Given any  $\sigma$, any $\delta>0$ and any set $S$ of extensions of $\sigma$ of
some fixed length $|\sigma|+n$ such that $|S|\geq (1-\delta)\cdot 2^n$,
there exists at least one string $\tau\in S$ such that
$M(\tau^{\ast})\leq M(\sigma)/(1-\delta)$ for all $\tau^{\ast}$ 
with $\sigma \subseteq \tau^{\ast} \subseteq \tau$.
\end{lem}
\begin{proof}
Towards a contradiction suppose that there 
exists no such string in $S$, and for each $\tau\in S$ 
let $\tau^{\ast}$ be the shortest initial segment extending $\sigma$ 
for which $M(\tau^{\ast})>M(\sigma)/(1-\delta)$. 
Then $S^{\ast}=\{ \tau^{\ast}\ |\ \tau\in S \}$ is a prefix-free set 
of strings. Since every element of $S$ has an initial segment in $S^{\ast}$ it follows that: 
\[
\sum_{\tau^{\ast}\in S^{\ast}} 2^{-|\tau^{\ast}|}\cdot M(\tau^{\ast})>
(1-\delta)\cdot 2^{-|\sigma|}\cdot  \frac{M(\sigma)}{1-\delta}=2^{-|\sigma|} \cdot M(\sigma)
\]
which contradicts Kolmogorov's inequality relative to $\sigma$.
\end{proof}
Note that $M(\sigma)/(1-\delta)=M(\sigma)+M(\sigma)\cdot \delta/(1-\delta)$,
so a small multiplicative amplification of the capital from $\sigma$ to $\tau$ can be translated into a small
additive increase in $M(\tau)-M(\sigma)$, as long as we keep $M(\sigma)$ under a fixed bound.
\begin{table}
\centering
\renewcommand{\arraystretch}{1.2}
\begin{tabular}{cll}
 \hline\cmidrule{1-3}
 $\sigma_n$ & &   {\small the $n$th initial segment of $X$ with approximations $\sigma_n[s]$}\\[0.1cm]
$s_n$ & &   {\small length of $\sigma_n$ 
according to the calculations in \S\ref{tP8HGMsW5Q}}\\[0.1cm]
$\epsilon_n$ & &   {\small appropriate value of the error $\epsilon$ of 
Lemmata \ref{lGR1K9yUT} and \ref{T2RhWOVdo} 
at level $n$, set as $2^{-n-5}$}\\[0.1cm]
$q_n$ & &   {\small bound on $K_V(X\restr_{s_n})/s_n$ set at $1/2+3/(n+2)$}\\[0.1cm]
$2^{-p_n}$ & &   {\small sufficient upper bound on $M(\sigma_{n-1})-M_t(\sigma_{n-1})$ for $\sigma_n[t]=\sigma_n$
(assuming $\sigma_{n-1}[t]=\sigma_{n-1}$)}\\[0.1cm]
\hline\cmidrule{1-3} 
\end{tabular}
\caption{Parameters for the proof of \eqref{hEsg1rmVV2}}
\vspace{-0.5cm}
\label{Sd48MuSJNg}
\end{table}
Once $\sigma_{n-1}$ has been chosen and a `fat' (\ie high probability) 
set of appropriate extensions
has been determined, 
Lemma  \ref{FDhethbip9} tells us that we will be able to 
choose $\sigma_n$ without increasing the capital of  $M$ by too much. 
The following fact follows from
Lemma  \ref{FDhethbip9} and the law of large numbers 
in Lemma \ref{PvVOtknZpc}.
\begin{lem}[Special extension]\label{lGR1K9yUT}
There exists a computable function $f$ such that if
$M$ is a non-negative martingale such that
$M(\lambda)\leq 1$, 
then for each $\epsilon \in (0,1)$, $\sigma\in 2^{<\Nat}$,
and $\ell>f(\epsilon)$ there exists
$\tau\succeq \sigma$ of length 
$\ell$ such that
$M(\rho)\leq M(\sigma)/(1-\epsilon)$ for all $\rho\in [\sigma,\tau]$ 
and the number of zeros (and hence, 1s) in $\tau$ after $\sigma$ is in 
$\big((1-\epsilon)(|\tau|-|\sigma|)/2, (1+\epsilon)(|\tau|-|\sigma|)/2\big)$.
\end{lem}
We   show that the potential for success for a 
separable strategy along the extensions of Lemma \ref{lGR1K9yUT} is limited.
\begin{lem}[Growth along special extensions]\label{T2RhWOVdo}
Let $(N_i)$, $(T_j)$ be computable families of 0-sided and 1-sided martingales respectively,
with finite total initial capital and consider the mixture $M=\sum_i N_i+\sum_j T_j$
with the approximations $M_s=\sum_{i<s} N_i+\sum_{j<s} T_j$.
Given $\epsilon,\sigma$, if $\tau$
is the extension of $\sigma$ of Lemma \ref{lGR1K9yUT} applied on $M_s$,
then for all $t>s$  if
$M_t(\sigma)-M_s(\sigma)<2^{-p}$ then 
$M_t(\tau)\leq  
M_s(\tau) + 2^{\delta \cdot (|\tau|-|\sigma|)-p}$,
where $\delta:=(1+\epsilon)/2$.
\end{lem}
\begin{proof}
For simplicity let $N^{\ast}=\sum_{i\in [s,t)} N_i$ and
$T^{\ast}=\sum_{i\in [s,t)} T_i$, so that
$M_t(\tau)\leq M_s(\tau) +N^{\ast}(\tau) +T^{\ast}(\tau)$,
and note that
$N^{\ast}$ is a 0-sided martingale while $T^{\ast}$
is a 1-sided martingale.
Between stages $s$ and $t$ there is at most $2^{-p}$ increase in $M(\sigma)$,
so  $N^{\ast}(\sigma)+T^{\ast}(\sigma)\leq 2^{-p}$. By the properties of $\tau$,
letting $\delta:=(1+\epsilon)/2$,
there exist at most $\delta\cdot (|\tau|-|\sigma|)$
many 0s between $\sigma$ and $\tau$ and the same is true of the 1s. Hence,
since $N^{\ast},T^{\ast}$ are single-sided, 
we have
$N^{\ast}(\tau)\leq N^{\ast}(\sigma)\cdot 2^{\delta\cdot (|\tau|-|\sigma|)}$
and
$T^{\ast}(\tau)\leq T^{\ast}(\sigma)\cdot 2^{\delta\cdot (|\tau|-|\sigma|)}$.
By adding these two, using the fact that $N^{\ast}(\sigma)+T^{\ast}(\sigma)\leq 2^{-p}$, we get
$M_t(\tau)\leq M_s(\tau) + 2^{-p}\cdot 2^{\delta\cdot (|\tau|-|\sigma|)}$.
\end{proof}

\subsection{Fixing the parameters for the proof of \eqref{hEsg1rmVV2}}\label{tP8HGMsW5Q}
Let  $(M_s)$ be a canonical approximation to $M$.
We will use Lemma \ref{T2RhWOVdo} for the definition of the sequence $(\sigma_i)$
that we discussed in \S\ref{96WqHtN5Iv}. For the approximations to 
$\sigma_n$ with $n>0$, we will apply Lemma \ref{T2RhWOVdo}
for the specific values $\epsilon_n=2^{-n-5}$ of $\epsilon$ and $p_n$ of $p$ (to be defined below), 
thus obtaining increasingly better
bounds for larger $n$. 
For each $n$ the segment $\sigma_n$ as 
well as its approximations will have a fixed length $s_n$
which we motivate and define as follows.
Suppose that $n>0$ and 
our choice of $\sigma_{n-1}$ has settled, but that now we 
are forced to choose a new value of $\sigma_n$, because the 
capital on some initial segment has increased by too much. 
What does Lemma \ref{T2RhWOVdo} tell us about the increase 
in capital, $2^{-p_n}$ say, that must have seen at $\sigma_{n-1}$ 
in order for this to occur? A bound for $p_n$ gives a corresponding 
bound on the number of times that $\sigma_n$ will have to be chosen: after $\sigma_{n-1}$ has settled
the approximation to the next initial segment $\sigma_n$ can change at most $2^{p_n}$ many times.
Overall, $\sigma_n$ can then change at most $2^{\sum_{i<n} p_i} \cdot 2^{p_n}$ many times, and
in order to satisfy \eqref{ivX9i8OUls}, at each of these changes we 
need to enumerate to the machine $V$ a description of
length $q_n\cdot s_n$. In order to keep the weight of these 
requests bounded, we will aim at keeping the total
weight of the requests corresponding to $\sigma_n$ bounded 
above by $2^{-n}$, for which it is sufficient that:
\begin{equation}\label{evywicXKuu}
2^{-s_nq_n} \cdot 2^{p_n}\cdot 2^{\sum_{i<n} p_i} <2^{-n}\iff
2^{p_n-s_nq_n}  <2^{-n-\sum_{i<n} p_i}\iff
s_nq_n-p_n>n+\sum_{i<n} p_i.
\end{equation}
By the bound given in Lemma \ref{T2RhWOVdo} 
in order for the growth of $\hat{M}(\sigma_n)$ at each length $s_n$
to be bounded above by $2^{-n-2}$, we need to set:
\begin{equation}\label{dQnx9IqQ4D}
p_n=s_n\cdot \delta_n+n+2
\hspace{0.3cm}\textrm{where $\delta_n:=(1+\epsilon_n)/2$}.
\end{equation}
By  Lemma \ref{T2RhWOVdo} it then follows that 
any growth of $M(\tau)$ by at least $2^{-n-2}$ for some 
$\tau$ with $\sigma_{n-1} \subseteq \tau \subseteq \sigma_n$, 
requires an increase of at least $2^{-p_n}$ in $M(\sigma_{n-1})$.  
Then $p_n-q_ns_{n}=n+2+s_{n}\cdot\left(\delta_n-q_n\right)$. 
By the definitions of $q_n,\epsilon_n$ we have
$\delta_n<q_n$, so \eqref{evywicXKuu} reduces to:
\[
s_{n}\cdot\left(q_n-\delta_n\right)> 2n+2+\sum_{i<n} p_i\iff
s_{n}\geq 
\frac{2n+2+ \sum_{i<n} p_i }{q_n-\delta_n}.
\]
Considering the bound of Lemma \ref{lGR1K9yUT} for the existence of
a special extension of $\sigma_{n-1}$, it suffices that:
\begin{equation}\label{uzDUjAkZU3}
s_{n}=\max\left\{
\frac{2n+2+ \sum_{i<n} p_i }{q_n-\delta_n},
f(\epsilon_n)\right\}.
\end{equation}
where $f$ is the computable function of  Lemma \ref{lGR1K9yUT}.

\subsection{Construction and verification for \eqref{hEsg1rmVV2}}\label{n6OP5Cpi6D}
We inductively define the approximations $\sigma_n[s]$ of $\sigma_n$ for all $n$, in stages $s$.
Let $\sigma_0[s]=\lambda$ for all $s$, $s_0=0$ and $\delta_n:=(1+\epsilon_n)/2$ for all $n$.  
The following notion incorporates 
the properties of Lemma \ref{lGR1K9yUT} in the framework of the construction and the specific values
of the parameters that were set in \S \ref{tP8HGMsW5Q}.
\begin{defi}[Special extensions]\label{HJuByjJ6sW}
At each stage $s+1$ and for each $n>0$ such that $\sigma_{n-1}[s]\de$ we say that $\tau$ 
is a special extension of $\sigma_{n-1}[s]$ if 
$|\tau|=s_n$, $\hat{M}_s(\tau)\leq \hat{M}(\sigma_{n-1})[s]/(1-\epsilon_n)$ and the number of
0s as well as the number of 1s between $\sigma_{n-1}[s]$ and $\tau$ is at most 
$\delta_n \cdot (s_n-s_{n-1})$.
\end{defi}
\begin{defi}[Attention]\label{QAbipYWtPW}
At stage $s+1$ the segment $\sigma_n$
requires attention if $n>0$ and either
%
$\sigma_n[s]\de$ and 
$\hat{M}_{s+1}(\sigma_n[s])> 2^{-1}+\sum_{i<n} 2^{-i-2}$,
or $\sigma_n[s]\un$.
\end{defi}
{\bf Construction for \eqref{hEsg1rmVV2}.} At stage $s+1$ pick the 
least $n\leq s$ such that $\sigma_n$ requires attention, if such exists.
 If $\sigma_n[s]\un$,
define $\sigma_n[s+1]$  to be the leftmost special extension  of $\sigma_{n-1}[s]$.
If $\sigma_n[s]\de$, set $\sigma_i[s+1]\un$ for all $i\geq n$. 
In any case,
let $k\leq s$ the least (if such exists) such that $\sigma_k[s+1]\de$ and 
$K_{V_{s}}(\sigma_k[s+1])>q_k \cdot s_k$,
and issue a $V$-description of $\sigma_k[s+1]$ of length $q_k \cdot s_k$.

{\bf Remark.} 
If at stage $s+1$ segment $\sigma_n[s+1]$ is newly defined 
as a special extension of $\sigma_{n-1}$ we have
\begin{equation}\label{Yvjr6QVlXm}
\hat{M}(\sigma_n)[s+1]\leq \hat{M}_{s+1}(\sigma_{n-1}[s])/(1-\epsilon_n)=\hat{M}_{s+1}(\sigma_{n-1}[s])+
\hat{M}_{s+1}(\sigma_{n-1}[s])\cdot \epsilon_n/(1-\epsilon_n).
\end{equation}
Since $\hat {M}(\sigma_{n-1})<1$ 
and $\epsilon_n/(1-\epsilon_n)< 2^{-n-2}$, condition \eqref{Yvjr6QVlXm} implies
\begin{equation}\label{GlIIlBhrRQ}
\hat{M}(\sigma_n)[s+1]\leq  \hat{M}(\sigma_{n-1})[s+1] + 2^{-n-2}.
\end{equation}
{\bf Verification of the construction for \eqref{hEsg1rmVV2}.}
By  Lemma \ref{FDhethbip9} 
we can always find a special extension as required in the first clause of the construction. 
In this sense, the construction of $(\sigma_n[s])$ is well-defined. 
In any interval of stages where $\sigma_{n-1}$ 
remains defined, successive values of $\sigma_n$ are lexicographically increasing.
It follows that 
each $\sigma_n[t]$ converges to a final value $\sigma_n$ such that $\sigma_n\prec \sigma_{n+1}$.
The real $X$ determined by the initial segments $\sigma_n$ is thus \lce and since 
\eqref{Yvjr6QVlXm} implies \eqref{GlIIlBhrRQ}, we have
 $\hat{M}(X\restr_n)<1$ for all $n$. 

It remains to show that the weight of $V$ is bounded above by 1.
Suppose that $\sigma_n$ gets newly defined at stage $s+1$ and at stage $t>s+1$ it becomes
undefined, while $\sigma_{n-1}[j]\de$ for all $j\in [s,t]$. 
Then  
$\hat{M}(\sigma_n)[s+1]\leq  \hat{M}(\sigma_{n-1})[s] +2^{-n-2}$.
Since $\sigma_n$ becomes undefined at stage $t$, we have
$\hat{M}_{t}(\sigma_n[s+1])> 2^{-1}+\sum_{i< n} 2^{-i-2}$.
By Lemma \ref{T2RhWOVdo} and \eqref{dQnx9IqQ4D}
it follows that  $\hat{M}_{t}(\sigma_{n-1}[s+1])-M(\sigma_{n-1})[s+1]>2^{-p_n}$.
Hence:
\[
\parbox{13.5cm}{during an interval of  stages  where $\sigma_{n-1}$ remains defined, 
$\sigma_n$ can take at most $2^{p_n}$ values}
\]
which means that the weight of the $V$-descriptions that we enumerate for strings of length $s_n$
is at most $2^{-s_nq_n} \cdot 2^{\sum_{i\leq n} p_i}$. 
By the definition of $s_n$ in \eqref{uzDUjAkZU3} and \eqref{evywicXKuu} this weight is bounded
above by $2^{-n}$. So the total weight of the descriptions that are enumerated
into $V$ is at most 1.

\section{Proof of Theorem \ref{1iotxpCAHo} and generalizations}\label{BBcmfYbEE}
It is possible to adapt the proof of Lemma  \eqref{hEsg1rmVV2} into
an effective construction for the proof of Theorem \ref{1iotxpCAHo}, which also
gives that the real $X$ is \lce 
For simplicity, we opt for a less constructive initial segment argument for the proof of
Theorem \ref{1iotxpCAHo}, which uses the facts we obtained in \S\ref{aRvZodD8FN}
in a modular way. The price we pay is that the constructed 
$X$ is no-longer \lce as in \eqref{hEsg1rmVV2}.
The following is the main tool for the proof of Theorem \ref{1iotxpCAHo}, where
$q_n$ has the same value as in \S\ref{aRvZodD8FN}.
\begin{lem}[Inductive property]\label{YSu9OrXiew}
There exists a \pf machine $V$ such that
for each $n> 0$,
$\sigma_0\prec\cdots\prec\sigma_{n-1}$
and $M=\sum_{j<x} N_j$, where each $N_j, j<x$ is a 
mixture of a computable family of strictly single-sided 
martingales with
\begin{equation}\label{W6AEs4G5UM}
\hat{M}(\sigma_{n-1})< 2^{-1}+ \sum_{i<n-1} 2^{-i-2}
\end{equation}
where $\sigma_{0}$ is the empty string, 
there exists
$\sigma_n\succ\sigma_{n-1}$ 
such that
\begin{equation}\label{hP2FnHWO9I}
K_V(\sigma_n)< |\sigma_n|\cdot q_{n}
\hspace{0.5cm} \textrm{and} \hspace{0.5cm}
\hat{M}(\sigma_n)\leq 
\left(2^{-1}+\sum_{i<n-1} 2^{-i-2}\right)
+2^{-n-2}.
\end{equation}
\end{lem}

\subsection{Proof of Theorem \ref{1iotxpCAHo} from Lemma \ref{YSu9OrXiew}}\label{fzQkvVAFVV}
Let $(M_j[s])$ be a (non-effective) list of all 
canonical approximations to \lce martingales $M_j$ with initial capital $<1$. This list includes
an approximation to each 
strongly \lce strictly single-sided 
martingale,
and by Lemma \ref{XXbFLuIodX}
it suffices to show the theorem with regard to the martingales in this list.
Using Lemma \ref{YSu9OrXiew}
we inductively define a sequence $(\sigma_i)$ of strings such that $\sigma_i\prec\sigma_{i+1}$ for each $i$,
$\sigma_0$ is the empty string,
and \eqref{hP2FnHWO9I} holds
for each $n>0$ and the separable \lce martingale 
\[
S_n := \sum_{j<n} 2^{-|\sigma_j|-j-2}\cdot M_j.
\]
We will ensure that for each $n$ we have
\begin{equation}\label{bquLE4fMd2}
K_V(\sigma_n)< |\sigma_n|\cdot q_{n}
\hspace{0.7cm} \textrm{and} \hspace{0.7cm}
\hat{S}_n(\sigma_n)\leq 2^{-1}+\sum_{i<n} 2^{-i-2}\leq 1
\end{equation}
where $V$ is the \pf machine of Lemma \ref{YSu9OrXiew}.
If we let $X$ to be the real defined by the initial segments $(\sigma_i)$,
then the second clause of \eqref{bquLE4fMd2} implies that for each $j,x\in\Nat$ we have
$M_j(X\restr_x)\leq 2^{|\sigma_j|+j+2}<\infty$ as required. 
Hence by the choice of $(M_j)$ and Lemma \ref{XXbFLuIodX}, 
no single-sided \lce martingale succeeds on $X$ and by Theorem \ref{LXn23yy81K}
the effective Hausdorff dimension is at least 1/2. The first clause of 
 \eqref{bquLE4fMd2} implies that the  effective Hausdorff dimension of $X$ is at most 1/2.
Hence  \eqref{bquLE4fMd2}
is sufficient for the proof of Theorem \ref{1iotxpCAHo}.

For the inductive use of Lemma \ref{YSu9OrXiew} in the construction, 
consider the following inequality for some $n$ 
(the equality is always true) which can be thought of as obtained at step $n$ from
Lemma \ref{YSu9OrXiew}:
\begin{equation}\label{yu2nRMn7Yq}
\hat{S}_n(\sigma_n)\leq 
\left(2^{-1}+\sum_{i<n-1} 2^{-i-2}\right)+2^{-n-2}=
\left(2^{-1}+\sum_{i<n} 2^{-i-2}\right)-2^{-n-2}
\end{equation}
and note that, since $M_n(\sigma_n)<2^{|\sigma_n|}$ we have 
$2^{-|\sigma_n|-n-2}\cdot M_n(\sigma_n)<2^{-n-2}$  so
\begin{equation}\label{4W8mhkopqL}
\textrm{if \hspace{0.3cm}\eqref{yu2nRMn7Yq} \hspace{0.3cm}holds then}
\hspace{0.5cm}
\hat{S}_{n+1}(\sigma_n)\leq 
2^{-1}+\sum_{i<n} 2^{-i-2}
\end{equation}
which is the hypothesis that is needed in order to apply  Lemma \ref{YSu9OrXiew} at step $n+1$
which will define $\sigma_{n+1}$.

{\bf Construction.} Let $\sigma_0$ be the empty string and let $V$ be the machine from
Lemma \ref{YSu9OrXiew}. For each $n>0$, inductively assume that
\eqref{W6AEs4G5UM} holds for $S_{n-1}$ 
and let $\sigma_n$ be an extension of $\sigma_{n-1}$ such that 
\eqref{hP2FnHWO9I} holds.

{\bf Verification.} First we show that the construction is well-defined.
Note that  \eqref{hP2FnHWO9I} holds for $n=0$. Assuming that $n>0$ and 
\eqref{hP2FnHWO9I} holds with $S_n$ in place of $M$, by \eqref{4W8mhkopqL} it follows that 
\eqref{W6AEs4G5UM} holds with $S_{n+1}$ in place of $M$ and with $n+1$ in place of $n$. 
Hence by Lemma \ref{YSu9OrXiew} at step $n+1$ there exists an extension $\sigma_{n+1}$
of $\sigma_n$ which satisfies \eqref{hP2FnHWO9I} with $n+1$ in place of $n$.
This concludes the justification that the construction is well-defined.
The construction, and in particular condition \eqref{hP2FnHWO9I} imposed on
the extensions, shows that \eqref{bquLE4fMd2} holds for each $n$. This concludes the verification of
the properties of the constructed sequence $(\sigma_i)$ and, as discussed above, the proof of 
Theorem \ref{1iotxpCAHo}.

\subsection{Preliminaries for the proof of Lemma \ref{YSu9OrXiew}}\label{QBksFMW19i}
The construction of $V$ of  Lemma \ref{YSu9OrXiew} will be computable, so for the proof of the lemma
we need to define an effective map which takes as an input $\eta$ a description (index) of $M$
and strings $\sigma_i, i<n$, and always outputs a {\em sufficiently small} part $V_{\eta}$ of $V$ 
(dealing with the specific input $\eta$) and an approximation $\sigma_n[s]$ such that
\begin{equation}\label{rfsRWL1sea}
\parbox{13.5cm}{if the input $\eta=(M,\sigma_i, i<n)$ meets the hypothesis of Lemma  \ref{YSu9OrXiew} then
$\sigma_n[s]$ converges to some $\sigma_n$ which satisfies the properties of the lemma with 
$V_{\eta}$ in place of $V$.}
\end{equation}
Since  Lemma  \ref{YSu9OrXiew} asks for a single machine $V$ that applies to all inputs,
we need to make sure that the special machines  $V_{\eta}$ are sufficiently small so that
there exists a machine $V$ with the property that $K_V$ is bounded by $K_{V_{\eta}}$ for all inputs
$\eta$. In order to express this property precisely, 
let $(N_i[s])$ an effective sequence (viewed as a double list of functions $\sigma\mapsto N_i(\sigma)[s]$)
of all canonical partial computable approximations of all (partial computable) 
mixtures of single-sided strategies.
The set $H$ of inputs is the set of all 
tuples $\eta=(i,\sigma_0,\dots, \sigma_{n-1})$
where $i\in\Nat$ is interpreted as an index in the list $(N_i[s])$,
and $\sigma_0\prec\cdots\prec\sigma_{n-1}$ is a chain of strings.
Let $g: H\to\Nat$ a one-to-one computable function so that
$\sum_{\eta \in H} 2^{-g(\eta)}<1$.

Given any $\eta \in H$ the map $\eta\mapsto (V_{\eta}, \sigma_n[s])$ will 
determine a  \pf machine $V_{\eta}$ such that
\begin{equation}\label{Be2NJ1XrhN}
\parbox{13cm}{for each $\eta \in H$,\hspace{0.3cm} 
$\wgtb{V_{\eta}}<2^{-g(\eta)}$\hspace{0.3cm}
\hspace{0.3cm}  so \hspace{0.5cm}$\sum_{\eta \in H} \wgt{V_{\eta}}<1$.}
\end{equation}
By \eqref{Be2NJ1XrhN}
we may define a {\em union} \pf machine $V$ such that $K_V(\rho)\leq K_{V_{\eta}}(\rho)$
for each $\rho$ and each $\eta\in H$. 
We have shown that
\[
\parbox{14.5cm}{for the proof of Lemma \ref{YSu9OrXiew} it suffices to define a computable map
$\eta\mapsto (V_{\eta}, \sigma_n[s])$ from $H$ to pairs of \pf machines and approximations
of strings, such that \eqref{rfsRWL1sea} and \eqref{Be2NJ1XrhN} hold.}
\]
The construction of the effective map  $\eta \mapsto (V_{\eta},\sigma_{n}[s])$ 
is a modification of the 
proof of \eqref{hEsg1rmVV2} in \S\ref{tP8HGMsW5Q},\S\ref{n6OP5Cpi6D}
and uses Lemma \ref{T2RhWOVdo} in the same way.
The parameters $q_n,\epsilon_n, p_n$ are as defined in  \S\ref{6E1fkvgPTH}.
Since here we have a special upper bound $2^{-g(\eta)}$ for the weight of $V_{\eta}$, we need
to re-calculate a suitable value for $s_n$, which is the required lower bound for the 
length of $\sigma_n$ of Lemma \ref{YSu9OrXiew}. 
Following \S\ref{tP8HGMsW5Q}, condition  \eqref{evywicXKuu}  becomes
\begin{equation}\label{lYREF6oaf3}
2^{-s_nq_n}\cdot 2^{p_n} <2^{-g(\eta)}\iff
s_nq_n - p_n>g(\eta).
\end{equation}
Given the definition of $p_n$ in \eqref{dQnx9IqQ4D} 
and arguing as in  \S\ref{tP8HGMsW5Q}, 
in order for the growth of $\hat{M}(\sigma_n)$ at each length $s_n$
to be bounded above by $2^{-n-3}$,
for \eqref{lYREF6oaf3} it suffices that
\[
s_{n}\cdot\left(q_n-\delta_n\right)> g(\eta)+n+3\iff
s_{n}\geq 
\frac{g(\eta)+n+3}{q_n-\delta_n}
\]
so it suffices to define the length of 
each approximation to $\sigma_n$ in the proof of Lemma \ref{YSu9OrXiew} by:
\begin{equation}\label{uzDUjAkZU3a}
s_{n}=s_{n}(\eta)= \max\left\{
\frac{g(\eta)+n+3}{q_n-\delta_n},
f(\epsilon_n)
\right\}.
\end{equation}
We also need the following simplified version of Definition \ref{HJuByjJ6sW} which will  be used in the
construction. 
\begin{defi}[Special extensions]\label{HJuByjJ6sWa}
At stage $s+1$ we say that $\tau$ 
is a special extension of $\sigma_{n-1}$ if
$|\tau|=s_n$  and 
it satisfies the properties of Lemma \ref{lGR1K9yUT} for $\epsilon:=\epsilon_n$, $\sigma:=\sigma_{n-1}$,
$p:=p_n$ and $s+1$ in place of $s$.
\end{defi}
It remains to define and verify the construction of the map $\eta\mapsto (V_{\eta}, \sigma_n[s])$.

\subsection{Construction and verification for the proof of Lemma \ref{YSu9OrXiew}}\label{RZE1YLbjLM}
Given $\eta\in H$, 
let $\sigma_j, j<n$ be the associated strings in $\eta$
in order of magnitude. 
For simplicity, let $(M_s)$ be the canonical partial computable \lce approximation given by $\eta$
and let $U:=V_{\eta}$. 
The following construction, on input $\eta$ produces
an effective enumeration $U_s$ of the \pf machine $U:=V_{\eta}$ and an effective
approximation $\sigma_n[s]$ of the string $\sigma_n$ (which may or may not converge).

{\bf Construction of $U,\sigma_n$ from $\eta$.}
At stage 0 we let $\sigma_n[0]\un$ and $U_0$ be empty.  
At stage $s+1$, do the following provided that $M_{s+1}$ is defined on all strings of length $s$,
and \eqref{W6AEs4G5UM} holds at stage $s+1$ (otherwise go to the next stage).
Check if one of the following holds:
\begin{enumerate}[\hspace{0.5cm}(i)]
\item $\sigma_n[s]\un$, there have been at most $2^{p_n}$ previous definitions of
$\sigma_n$ in previous stages, and there exists a special extension  of $\sigma_{n-1}$.
\item $\sigma_n[s]\de$\hspace{0.3cm} and \hspace{0.3cm} 
$\hat{M}_{s+1}(\sigma_n[s])> 2^{-1}+\sum_{i\in [j,n-1)} 2^{-i-2} +2^{-n-3}$.
\end{enumerate}
If (i) holds,
define $\sigma_n[s+1]$  to be the leftmost special extension  of $\sigma_{n-1}$ 
as per Definition \ref{HJuByjJ6sWa}.
If (ii) holds, set $\sigma_n[s+1]\un$. 
In any case, if
$\sigma_n[s+1]\de$ and $K_{U_{s}}(\sigma_n[s+1])>q_n \cdot |\sigma_n[s+1]|$,
issue a $U$-description of $\sigma_n[s+1]$ of length $q_n \cdot |\sigma_n[s+1]|$.

{\bf Remark.} 
If $\sigma_n[s+1]$ is newly
defined as a special extension of $\sigma_{n-1}$, by 
Definition \ref{HJuByjJ6sWa} we have that
\begin{equation}\label{Yvjr6QVlXma}
\hat{M}_{s+1}(\sigma_n)[s+1]\leq \hat{M}_{s}(\sigma_{n-1})/(1-\epsilon_n)=
\hat{M}_{s+1}(\sigma_{n-1})+
\hat{M}_{s+1}(\sigma_{n-1})\cdot \epsilon_n/(1-\epsilon_n).
\end{equation}
By \eqref{W6AEs4G5UM} referenced at stage $s+1$, we have
$\hat {M}_{s+1}(\sigma_{n-1})<1$ so 
by $\epsilon_n/(1-\epsilon_n)< 2^{-n-3}$ 
condition \eqref{Yvjr6QVlXma} implies
\begin{equation}\label{GlIIlBhrRQa}
\hat{M}_{s+1}(\sigma_n[s+1])\leq  \hat{M}_{s+1}(\sigma_{n-1})+ 2^{-n-3}.
\end{equation}

{\bf Verification of the constructing of $U,\sigma_n$ from $\eta$.}

First we show that
the weight of $U$ is bounded above by $g(\eta)$.
In this argument we do not assume anything about the input $\eta$, the associated approximation
$(M_s)$, or the convergence
of the approximations $(\sigma_n[s])$.
Clause (i) of the construction enforces that 
\begin{equation}\label{RzlXpDiFH}
\parbox{10cm}{the approximation to $\sigma_n$ can change at most $2^{p_n}$ many times.}
\end{equation}
This is the assumption we used in our calculations of \eqref{lYREF6oaf3}
and \eqref{uzDUjAkZU3a},
which we can now use to derive the bound on the weight of $U$ 
based on the values of $p_n, s_n, q_n,\epsilon_n$ that we set.
By \eqref{RzlXpDiFH}
and since $|\sigma_n[s]|= s_n$,
the weight of the $U$-descriptions that we enumerate for 
the approximations to $\sigma_n$
is at most $2^{-s_nq_n} \cdot 2^{p_n}$.
The definition of $s_n$ in \eqref{uzDUjAkZU3a} 
and  \eqref{lYREF6oaf3} imply that the above bound is at most 
$2^{-g(\eta)}$ as required.

It remains to
show that in the case that if $(M_s)$ is a total computable canonical approximation 
to a single-sided mixture $M$
such that \eqref{W6AEs4G5UM} holds,
the construction will produce an approximation
$\sigma_n[s]$ which converges to a string $\sigma_n$ after finitely many stages, such that
the second clause of \eqref{hP2FnHWO9I} holds.
Suppose that $\sigma_n$ gets (re)defined at stage $s+1$ and at stage $t>s+1$ it becomes
undefined.  By \eqref{GlIIlBhrRQa} and \eqref{W6AEs4G5UM} we have
\begin{equation}\label{43v25y1h24}
 \hat{M}_{s+1}(\sigma_n[s+1])\leq   2^{-1}+ \sum_{i\in [j,n-1)} 2^{-i-2} + 2^{-n-3}.
\end{equation}
Since $\sigma_n$ becomes undefined at stage $t$, we have
$\hat{M}_{t}(\sigma_n[s+1])> 2^{-1}+\sum_{i< n-1} 2^{-i-2} +2^{-n-2}$.
 By \eqref{GlIIlBhrRQa},\eqref{43v25y1h24}, \eqref{W6AEs4G5UM}
and an application of  Lemma \ref{T2RhWOVdo} for $\epsilon:=\epsilon_n$, 
$p:=p_n$, $\sigma:=\sigma_{n-1}$ and $\tau:=\sigma_n[s+1]$,
it follows that  $\hat{M}_{t}(\sigma_{n-1})-M_{s+1}(\sigma_{n-1})>2^{-p_n}$.
Since the latter event can occur at most $2^{p_n}$ many times, 
we have shown that if $\sigma_n$ is newly defined 
by the construction at some stage $s+1$ and this is the $2^{p_n}$-th such definition during the construction,
then it will never be undefined again, \ie $\sigma_n[t]\de$ for all $t>s$.
In particular, the second clause of (i) in the construction (regarding the number of previous definitions
of $\sigma_n$) can never block the redefinition of $\sigma_n$, subject to the other two condition holding.
Given this fact, and Lemma \ref{lGR1K9yUT} which concerns the existence of special extensions, 
it is not possible that 
$\sigma_{n}$ is undefined for co-finitely many stages; in other words, 
{\em for each $s_0$ there exists $s>s_0$ such that $\sigma_n[s]\de$}.
%
Since $(M_s)$ is a \lce approximation, by \eqref{W6AEs4G5UM} and \eqref{GlIIlBhrRQa}
successive values of $\sigma_n[s]$ during redefinitions of $\sigma_n$  will be
lexicographically increasing, so
$\sigma_n[s]$ converges to a final value $\sigma_n$ such that $\sigma_{n-1}\prec \sigma_{n}$.
Since \eqref{Yvjr6QVlXma} implies \eqref{GlIIlBhrRQa}, we have
that the second part of \eqref{hP2FnHWO9I} holds, as required.
This concludes the proof of  \eqref{rfsRWL1sea} and \eqref{Be2NJ1XrhN} hence,
as explained in \S\ref{QBksFMW19i}, the proof of Lemma \ref{YSu9OrXiew}.

\subsection{Generalization to decidably-sided strategies}\label{meVVbZop96}
We adapt argument of \S\ref{fzQkvVAFVV}-\S\ref{RZE1YLbjLM} in order to prove 
the following analogue of Theorem \ref{1iotxpCAHo}.
\begin{thm}\label{X2gQ1AwUza}
There exists a real $X$ of effective Hausdorff dimension $1/2$
such that no decidably-sided strongly \lce martingale (or mixture of a computable family of $f$-sided
strategies for some computable $f$) succeeds on it.
\end{thm}
We need the following an analogue of Lemma \ref{YSu9OrXiew}
for decidably-sided \lce martingales.
\begin{lem}[Inductive property for decidably-sided sums]\label{baLoQ2YMz}
There exists a \pf machine $V$ with the property that
for each $n> 0$,
chain of strings $\sigma_0\prec\cdots\prec\sigma_{n-1}$, computable prediction 
functions $f_i, i<n$ and $M=\sum_{i<n} M_i$, where each $M_i, i<n$ is a mixture of a computable family 
of $f_i$-decidably sided
martingales
with canonical approximation such that 
\begin{equation*}
\hat{M}(\sigma_{n-1})< 2^{-1}+ \sum_{i\in [j,n-1)} 2^{-i-2}
\end{equation*}
where $\sigma_{0}$ is the empty string, 
there exists
$\sigma_n\succ\sigma_{n-1}$ 
such that
\begin{equation*}
K_V(\sigma_n)< |\sigma_n|\cdot q_{n}
\hspace{0.5cm} \textrm{and} \hspace{0.5cm}
\hat{M}(\sigma_n)\leq 
\left(2^{-1}+\sum_{i\in [j, n-1)} 2^{-i-2}\right)
+2^{-n-2}.
\end{equation*}
\end{lem}
Theorem \ref{X2gQ1AwUza} follows from Lemma \ref{baLoQ2YMz}
by the argument of \S\ref{fzQkvVAFVV} 
which derived Theorem \ref{1iotxpCAHo} from Lemma \ref{YSu9OrXiew}.
The only difference is that here $(M_j[s])$  is a list of all
canonical approximations to mixtures of decidably-sided martingales 
whose initial capital is less than 1. 
For the proof of Lemma \ref{baLoQ2YMz} we need to obtain 
analogues of the key facts from \S\ref{aRvZodD8FN} for  the case of 
decidably-sided martingales.
We start with the
following analogue of Lemma \ref{lGR1K9yUT}, which follows by a direct application of 
Lemma  \ref{FDhethbip9} to the law of large numbers 
in Lemma \ref{PvVOtknZpc}, applied to the intersection of finitely many events.
\begin{lem}[Special extensions for decidably sided]\label{h5fMyKh9Tb}
There exists a computable  $g$ such that for each
$\epsilon \in (0,1)$, $\sigma\in 2^{<\Nat}$, $n>0$,
and $(M_j, f_j)$, $j<n$ where each $f_j$ is a
prediction function and
$M_j$ is an $f_j$-sided martingale with $M_j(\lambda)\leq 1$,
and each $\ell>g(\epsilon,n)$,
there exists $\tau\succeq \sigma$ of length 
$\ell$ such that for each $j<n$,
\[
\parbox{13cm}{the number of correct $f_j$-predictions in $[\sigma,\tau]$
is in  $\Big((1-\epsilon) (|\tau|-|\sigma|)/2, (1+\epsilon) (|\tau|-|\sigma|)/2\Big)$}
\]
and $M_j(\rho)[s]\leq M_j(\sigma)[s]/(1-\epsilon)$ for all $\rho$ 
with $\sigma \subseteq \rho \subseteq \tau$.
\end{lem}
Now we may obtain the required analogue of Lemma  \ref{T2RhWOVdo}.
\begin{lem}[Growth along special extension for decidably sided]\label{UZnKI3ndd}
Let $(M_j, f_j),  j<n$ 
be as in Lemma \ref{h5fMyKh9Tb}, let
$M_j[s]$  be canonical approximations of $M_j$,
and define $N:=\sum_{j<n} M_j$ and $N_s:=\sum_{j<n} M_j[s]$.
Given $\epsilon>0,p,s\in\Nat, \sigma\in\twomel$, 
if $\tau$ is the extension of
$\sigma$ given by Lemma \ref{h5fMyKh9Tb},
then for all $t>s$: 
\[
N_t(\sigma)-N_s(\sigma)<2^{-p}
\Rightarrow
N_t(\tau)\leq  
N_s(\tau) + 
2^{\delta \cdot |\tau| -p}
\]
where $\delta:=(1+\epsilon)/2$.
\end{lem}
The proof of Lemma \ref{UZnKI3ndd} for the special case where $N$ is itself a mixture of $f$-sided martingales
for a computable $f$
is entirely analogous to the proof of Lemma  \ref{T2RhWOVdo} which refers to single-sided martingales,
with the difference that Lemma \ref{h5fMyKh9Tb} is used in place of Lemma \ref{lGR1K9yUT}.
The case where $N$ is the sum of finitely many such mixtures (with distinct prediction functions $f_j$)
follows from the special case in the same way that 
the separable case of Lemma  \ref{T2RhWOVdo} follows from the special case of a single-sided martingale
(recall the first paragraph of the proof of Lemma  \ref{T2RhWOVdo}).

It remains to show that a straightforward adaptation
of the argument in \S\ref{QBksFMW19i}, \S\ref{RZE1YLbjLM} 
proves Lemma \ref{baLoQ2YMz}.
The entire set-up of \S\ref{QBksFMW19i} remains the same, including the parameter values and
Definition \ref{HJuByjJ6sWa} which is later used in the construction, with the exception that
instead of Lemma \ref{lGR1K9yUT} we use Lemma \ref{UZnKI3ndd}.
The construction of the required map in \S\ref{RZE1YLbjLM} remains exactly the same, except that
the updated version of Definition \ref{HJuByjJ6sWa} of special extensions is used
(based on Lemma \ref{UZnKI3ndd} instead of  Lemma \ref{T2RhWOVdo}).
The verification of the construction in \S\ref{RZE1YLbjLM}  also remains the same, except that
the reference to
Lemma \ref{T2RhWOVdo}  is replaced with a reference to Lemma \ref{UZnKI3ndd}.
This concludes the proof of  Lemma \ref{baLoQ2YMz} and, as explained above, the proof of
Theorem \ref{X2gQ1AwUza}.

\section{Conclusion and some questions}\label{7FjYtfZZf}
We have studied the strength of monotonous strategies, which bet constantly 
on the same outcome (single-sided martingales) 
or bet on a computable outcome (decidably-sided martingales).
In the case of computable strategies we have seen that they are as strong as
the unrestricted strategies, while in the case of uniform effective mixtures of strategies 
(strongly \lce martingales) they are significantly weaker.
On the other hand, for casino sequences of effective Hausdorff dimension less than 1/2,
successful \lce strategies can be replaced by successful uniform effective mixtures of 
single-sided strategies.

{\bf Limitations of the present work and open problems.}
Our main negative results, Theorems \ref{1iotxpCAHo} and \ref{X2gQ1AwUza},
rely on two main properties: (a) the given strategies are martingales and not
merely supermartingales; (b) the given monotonous martingales are not merely \lce but 
strongly l,c.e., \ie are assumed to have \lce wagers. Restriction (a) relates to the non-interchangeability 
between martingales and supermartingales under monotonousness, as discussed in \S\ref{6E1fkvgPTH};
the main interest on (a) is the connection with a problem of Kastermans, which we briefly discuss below.
Perhaps most significant is restriction (b), which rests on the difference between
{\em mixtures of computable families of monotonous strategies} on the one hand, and
{\em monotonous mixtures of 
computable families of strategies} on the other. The difference in these two
approaches of  combining monotonousness with computable enumerability of strategies,
described as uniform and non-uniform in (i), (ii) of \S\ref{VYuXf7zTQg}
respectively, relies on whether the {\em intermediate bets}
witnessed by a computable observer with access to the approximation of the strategy
are monotonous or not.
Our main open question is whether (b) is essential for Theorems \ref{1iotxpCAHo} and \ref{X2gQ1AwUza}:
\begin{equation}\label{oqRTjUrctz}
\parbox{14cm}{{\bf Question:} If a \lce martingale succeeds on $X$, does there exist a \lce 
single-sided strategy (\ie a single-sided martingale $M$ which is the 
mixture of a computable family of strategies) which succeeds on $X$?}
\end{equation}
Equivalently, we can ask if the standard notion of algorithmic randomness, \ml randomness,
can be defined with respect to single-sided  \lce martingales.
A third limitation (c) in Theorem \ref{X2gQ1AwUza} is the assumption, included in 
Definition \ref{C7kLPhxcBb}, that the prediction functions $f$ are {\em total computable}
and not merely partial computable, allowing the possibility of partiality on states $\sigma$
where the wager is 0. Such a generalization would formalize a notion of {\em partially decidably-sided}
strategies, which cannot be dealt with by the argument in the proof of Theorem \ref{X2gQ1AwUza}.

{\bf Relation to a problem of Kastermans.}
Consider the case of \lce supermartingales that are
partially decidably-sided, according to the above discussion; 
such strategies are known as {\em kastergales}, see  \citep[\S 7.9]{rodenisbook}).
Kastermans, as reported in \citep{Downey:2012:RCM:2367234} and \citep[\S 7.9]{rodenisbook} 
asked whether there exists a sequence where all kastergales are bounded, 
but some computably enumerable
strategy succeeds. 
A simple negative answer to this question would be that
for every real $X$ where a \lce martingale succeeds, there exists
a single-sided, or even just decidably-sided martingale which succeeds on $X$.
First, note that a positive answer to \eqref{oqRTjUrctz} would give a very simple and
elegant negative answer to Kastermans' question.
In the same fashion, Theorem \ref{X2gQ1AwUza} can be viewed as a partial negative answer to
Kastermans' question. Then limitations (a), (b) and (c) of our methods discussed above are
the obstacles in extending our partial answer to a full negative answer  to Kastermans' question.

%

\bibliographystyle{abbrvnat}
\bibliography{gratin}
\end{document}